\documentclass[12pt]{article}

\usepackage{amsfonts,amssymb,amsbsy,amsmath,latexsym,graphicx}
\usepackage{natbib}
\usepackage{amsthm}
\usepackage{dsfont,bbding} 
\usepackage{subfigure,color}
\usepackage[margin=1in]{geometry}
\usepackage{braket}
\usepackage{hyperref}

\graphicspath{{./}}

\title{Variance bounding of delayed-acceptance kernels} 
\author{Chris Sherlock$^1$ and Anthony Lee$^2$} 
\date{{\small $^1$Department of Mathematics and Statistics, Lancaster
    University, Lancaster, LA1 4YF, UK; ORCID: 0000-0002-2429-3157; c.sherlock@lancaster.ac.uk.\\
$^{2}$School of Mathematics, Fry Buiding, University of Bristol, Bristol, BS8 1UG, UK; ORCID: 0000-0001-7765-0616.\\}}

\newtheorem{theorem}{Theorem}
\newtheorem{proposition}{Proposition}
\newtheorem{lemma}{Lemma}
\newtheorem{corollary}{Corollary}
\theoremstyle{definition}
\newtheorem{definition}{Definition}

\newtheorem{example}{Example}

\newcommand{\pihat}{\hat{\pi}}
\newcommand{\qhat}{\hat{q}}

\newcommand{\md}{\mbox{d}}

\newcommand{\Prob}[1]{\mathbb{P}\left({#1}\right)}
\newcommand{\Probs}[2]{\mathbb{P}_{{#1}}\left({#2}\right)}
\newcommand{\Expects}[2]{\mathbb{E}_{{#1}}\left[{#2}\right]}
\newcommand{\Norm}[1]{\left|\left|{#1}\right|\right|}
\newcommand{\Abs}[1]{\left|{#1}\right|}
\newcommand{\cip}{\stackrel{p}{\rightarrow}}

\newcommand{\setA}{\mathcal{A}}
\newcommand{\setB}{\mathcal{B}}
\newcommand{\setC}{\mathcal{C}}
\newcommand{\setD}{\mathcal{D}}
\newcommand{\setF}{\mathcal{F}}
\newcommand{\setM}{\mathcal{M}}
\newcommand{\setX}{\mathcal{X}}
\newcommand{\Bbar}{\overline{B}}
\newcommand{\asvar}{\mbox{\textsf{var}}}
\newcommand{\PDA}{\tilde{\mathsf{P}}}
\newcommand{\PMH}{\mathsf{P}}
\newcommand{\alphaDA}{\tilde{\alpha}}
\newcommand{\alphaMH}{\alpha}
\newcommand{\alphabarDA}{\overline{\widetilde{\alpha}}}
\newcommand{\alphabarMH}{\overline{\alpha}}
\newcommand{\rDA}{\mathsf{r}}
\newcommand{\rMH}{\mathsf{r}}

\begin{document}
\maketitle

\begin{abstract}
A delayed-acceptance version of a Metropolis--Hastings algorithm can
be useful for Bayesian inference when it is computationally expensive
to calculate the true posterior, but a computationally cheap
approximation is available;
the delayed-acceptance
kernel targets the same posterior as its associated ``parent'' Metropolis-Hastings
kernel. Although the asymptotic variance of the ergodic average of any functional of the delayed-acceptance chain
cannot be less than that obtained using its parent, the
average computational time per iteration can be much smaller and so 
for a given computational budget the delayed-acceptance kernel can be
more efficient.  

When the asymptotic variance of the ergodic averages of all $L^2$ functionals of the chain are
finite,  the kernel is said to be variance bounding. 
It has recently been noted that a delayed-acceptance kernel need not
be variance bounding even when its parent 
is.  
We provide sufficient conditions for inheritance:
for non-local algorithms, such as the independence sampler,
 the discrepancy between the log density of the approximation and that of the truth should be bounded; for local
algorithms, two alternative sets of conditions are
provided.
 
As a by-product of our initial, general result we also supply sufficient
conditions on any pair of proposals such that, for any shared target distribution, if a Metropolis-Hastings
kernel using one of the proposals is variance bounding
then so is the Metropolis-Hastings kernel using the other proposal.
\end{abstract}


\textbf{AMS:}Primary: {60J10}; Secondary: {65C40;47A10} 

\textbf{Declarations:} funding - none; conflicts of interest - none; availability of data and material - n/a; code availability - R code to produce the plots in Section 5 is available from \url{https://chrisgsherlock.github.io/Research/publications.html}.

\section{Introduction} 

\label{sect.intro}
The Metropolis-Hastings (MH) algorithm is widely used
to approximately compute expectations with respect to
complicated high-dimensional posterior distributions
\cite[e.g.][]{Gilks/Richardson/Spiegelhalter:1996, MCMChandbookChOne}. The algorithm requires
that it be possible to evaluate point-wise the density of the distribution of interest (throughout this article, all densities are with respect to Lebesgue measure) 
up to an arbitrary constant of
proportionality.

In many problems the target posterior density is computationally expensive
to evaluate. When a computationally-cheap approximation, or surrogate,
is available, the delayed-acceptance Metropolis-Hastings
(DAMH) algorithm 
\cite[][also known as the two-stage algorithm, and a special case of the surrogate-transition method]{liu2001monte,ChristenFox:2005,MCMCHandbookDA}
leverages the surrogate to produce a new Markov chain that still targets the original distribution of interest.
A first `screening' stage substitutes the surrogate density for 
the true density in the standard formula for
the MH acceptance probability; proposals which fail at this stage
are discarded. Only proposals that pass the first stage are
considered in the second `correction' stage, where it is necessary to
evaluate the true posterior density at the proposed value.

Delayed acceptance (DA) algorithms
have been applied in a variety of settings with the approximate density obtained in a variety of different ways, for example:
 a coarsening of a numerical
grid in Bayesian inverse problems
\cite[]{ChristenFox:2005,Moulton/etal:2008,cui2011bayesian}, subsampling from
big-data \cite[]{PayneMallick:2014,BGLR:2015,Quiroz:2015}, a tractable
approximation to a
stochastic process \cite[]{Smith:2011,GolightlyHendersonSherlock:2013}, or
a direct, nearest-neighbour approximation to the truth using previous
values \cite[]{SGH:2015}.

For a Markov kernel, $P$, with a stationary distribution of $\pi$, and an associated chain $\{X_t\}_{t=1}^\infty$, the
\emph{asymptotic variance} of any functional, $h$, is defined to be
\[
\asvar(h,P):=\lim_{n\rightarrow \infty} n\mbox{Var}\left[\frac{1}{n}\sum_{i=1}^n h(X_i)\right],
\]
where $X_1\sim \pi$. A lower asymptotic variance is thus associated,
in practice, with a
greater accuracy in estimating $\Expects{\pi}{h(X)}$ using a
 realisation of length $n>>1$ from the distribution of the chain.
In terms of the asymptotic variance of any functional of the chain, 
the DAMH kernel cannot be more efficient than the parent
MH kernel; however the
computational cost per iteration is, typically, reduced considerably. The
almost-negligible computational cost of the screening stage also, typically, facilitates proposals that
have a larger chance of being rejected than the MH proposal, but where
the pay-off on acceptance is so much larger that the expected
overall movement per unit of time increases. When efficiency is measured in terms
of effective samples per second, gains of over an order of magnitude
have been reported \cite[e.g.][]{GolightlyHendersonSherlock:2013}.

A Markov kernel $P$ with a stationary distribution of $\pi$ is termed
\textit{variance bounding} if $\asvar(h,P)<\infty$ for all 
$h\in L^2(\pi)$, the Hilbert space of functions that are
square-integrable with respect to $\pi$. Equivalently there exists $K<\infty$ such that $\asvar(h,P)\le K \mbox{Var}_{\pi}[h]$ for all such $h$.
 This property was named and 
studied in \cite{RobRos:2008}, where it was shown to be equivalent to the existence of a `usual'
central limit theorem (CLT); that is, a CLT where the limiting
variance is the asymptotic variance.

Intuitively, the variance-bounding property embodies desirable behaviour for a chain started at equilibrium. In practice, the chain is not started at equilibrium, but asymptotically the bias that results from this is negligible compared with the variance. An alternative natural requirement is that the chain converge to equilibrium geometrically quickly (rather than, say, polynomially quickly). 
A Markov chain kernel, $P$, with stationary distribution $\pi$ is
\emph{geometrically ergodic} \cite[e.g.][Chapter 15]{RnR1997,RobRos:2004,MeynTweedie1993} if there exist 
$\rho>0$ and $M:\setX\rightarrow [0,\infty)$ 
that is finite $\pi$-almost everywhere, such that
\[
  |P^n(x,\setA)-\pi(\setA)|\le M(x)\rho^n
\]
 for all
$\setA\in\setF, ~x\in\setX$ and $n\in\mathbb{N}$, where $P^n$ denotes the $n$-step transition kernel.

Although the motivations behind the definitions of variance bounding and geometric ergodicity, mixing at equilibrium and convergence to equilibrium, are quite different, for a large class of algorithms, including those studied in this article, these two properties are very closely linked as we will describe in Section \ref{sec.KeyTerminImplic}. Indeed, for delayed-acceptance algorithms, under weak conditions the two properties are equivalent (see Proposition \ref{prop.GEtonoGEbad}). 

Theoretical properties of the efficiency of delayed-acceptance algorithms have
been studied in \cite{BGLR:2015}, \cite{STG:2015} and \cite{FranksVihola2020}. The first contribution from \cite{BGLR:2015} is an example delayed-acceptance algorithm which fails to inherit geometric ergodicity from its parent Metropolis-Hastings algorithm (see Example \ref{example.BGLR} in
Section \ref{sect.algorithms} of this article); a simple sufficient condition for inheritance of geometric ergodicity, uniformly good behaviour of the ratios $\rDA_1$ and $\rDA_2$ that we define in \eqref{eqn.define.rrr}, is also supplied. Finally, an idealised setting where the cheap approximation is perfectly accurate is explored to obtain tuning guidelines for $\lambda$ in the delayed-acceptance random walk Metropolis algorithm. \cite{STG:2015} examines this tuning issue further, proving a limiting diffusion for the first component of the delayed-acceptance Markov chain, and providing robust tuning guidelines that account for the error in the cheap approximation; the article then extends these guidelines to the pseudo-marginal version of the algorithm. Finally, \cite{FranksVihola2020} compares the asymptotic variance of a general pseudo-marginal delayed-acceptance algorithm with the variance of an algorithm that applies importance-sampling to the output of an MCMC algorithm targeting the cheap approximation directly. 

Using our Proposition \ref{prop.GEtonoGEbad}, the lack of inheritance of geometric ergodicity in the example in \cite{BGLR:2015} is equivalent to a lack of inheritance of the variance bounding property:  
even though the asymptotic variance using the parent MH kernel is
finite for all $h \in L^2(\pi)$, there exist $h\in L^2(\pi)$ 
for which the asymptotic variance using the DA kernel is infinite. For such $h$, estimated quantities such as effective sample size \cite[e.g.][]{Hoff2009} are invalid, and consequent, standard CLT-based intuitions about the sizes of typical errors in estimates of $\mathbb{E}_{\pi}[h]$ from the chain \emph{do not hold}.

We investigate the conditions under
which a DAMH kernel inherits 
variance bounding from its MH parent and, as a by product, discover 
conditions under which two different proposals produce MH kernels that
are equivalent in terms of
whether or not they are variance bounding. Section \ref{sect.backg}
provides the background and two motivating examples, while Section \ref{sect.generic}
provides some key definitions, a general inheritance result applicable to all
propose-accept-reject kernels, and sufficient conditions for
variance-bounding equivalence between two
Metropolis-Hastings proposals. Section
\ref{sect.results} contains our results for standard DA algorithms with further illustrative
examples, and includes parent MH algorithms where the proposal depends upon the
form of the density, so
that the proposal for a computationally cheap DA kernel would
naturally depend on the surrogate. Numerical experiments are performed in Section \ref{sec.numerical} and the article concludes with a discussion. All proofs are deferred to Appendix \ref{sect.proofs}.
 
\section{Background, notation and motivation}
\label{sect.backg}
Throughout this article all Markov chains are assumed to be on a 
statespace $(\setX,\setF)$, with $\setX\subseteq \mathbb{R}^d$
Lebesgue measurable, and
$\setF$ the $\sigma$-algebra of all Lebesgue-measurable sets in
$\setX$. The target and surrogate distributions are denoted by $\pi$ and $\pihat$, respectively, and they  are assumed to have densities of $\pi(x)$ and $\pihat(x)$ with
respect to Lebesgue measure. 

\subsection{Metropolis-Hastings and delayed-acceptance kernels}
The Metropolis-Hastings kernel has a proposal density 
$q(x,y)$ and an acceptance probability $\alphaMH(x,y)=1\wedge \rMH(x,y)$
where
\begin{equation}
  \label{eqn.define.r}
\rMH(x,y) := \frac{\pi(y)q(y,x)}{\pi(x)q(x,y)}.
\end{equation}
With $\alphabarMH(x):=\int \alpha(x,y) q(x,y) \md y$,
the Metropolis-Hastings (MH) kernel is then
\begin{equation}
\label{eqn.PMH}
\PMH(x,\md y):=q(x,y)\md y~
\alphaMH(x,y)+[1-\alphabarMH(x)]\delta_{x}(\md y).
\end{equation}
An iteration of the corresponding MH algorithm proceeds from a current value, $x$, to the next value, $y$, as follows. A value $x'$ is sampled from the distribution with a density of $q(x,x')$. With a probability of $\alpha(x,x')$, $y\leftarrow x'$, else $y\leftarrow x$.

Now, suppose that we have an approximation, $\pihat(x)$, to
$\pi(x)$. The standard delayed-acceptance
kernel uses the same proposal, $q(x,y)$, but has an
acceptance probability of 
$\alphaDA(x,y)=[1\wedge \rDA_1(x,y)][1\wedge \rDA_2(x,y)]$, where
\begin{equation}
  \label{eqn.define.rrr}
\rDA_1(x,y):=\frac{\pihat(y)q(y,x)}{\pihat(x)q(x,y)}
~~~\mbox{and}~~~
\rDA_2(x,y):=\frac{\pi(y)/\pihat(y)}{\pi(x)/\pihat(x)}.
\end{equation}
With $\alphabarDA(x):=\int \alphaDA(x,y) q(x,y) \md y$, the
delayed-acceptance (DA) kernel is
\begin{equation}
\label{eqn.PDA}
\PDA(x,\md y):=q(x,y)\md y
~\alphaDA(x,y)+[1-\alphabarDA(x)]\delta_{x}(\md y).
\end{equation}
An iteration of the corresponding DA algorithm proceeds from a current value, $x$, to a next value, $y$, as follows.
\begin{description}
\item[Stage One:] A value $x'$ is sampled from the distribution with a density of $q(x,x')$. With a probability of $1\wedge \rDA_1(x,x')$ the algorithm proceeds to Stage Two, else $y\leftarrow x$.
  \item[Stage Two:] With a probability of $1\wedge \rDA_2(x,x')$, $y\leftarrow x'$, else $y\leftarrow x$.
\end{description}
Now, $\alphaDA(x,y)\le \alphaMH(x,y)$, and so
 $\asvar(h,\PDA)\ge \asvar(h,\PMH)$ for each $h\in L^2(\pi)$ \cite[]{Peskun:1973,Tierney:1998}. At first glance this might suggest that the DA algorithm is never worthwhile; however for any proposal that is rejected at Stage One there is no need to complete the expensive calculation of $\pi(x')$ that is required at every iteration of the MH algorithm and in Stage Two of the DA algorithm. As mentioned in the Introduction, for a fixed computational time, the decreased average computational cost per iteration, and alterations of any tuning parameters to take advantage of this, can lead to a DA algorithm where the variance of an estimator can be over an order of magnitude smaller than that of the MH algorithm.

Since $\asvar(h,\PDA)\ge \asvar(h,\PMH)$, if $\PDA$ is variance
bounding then so is $\PMH$; however
it is feasible that $\PMH$ may be variance bounding while $\PDA$ is not.

\subsection{Key terminology, equivalences and implications}
\label{sec.KeyTerminImplic}
The MH and DA kernels are both \emph{reversible} with respect to the target. A kernel $P$ is reversible with respect to a distribution $\pi$ iff for all $\setA\in \setF$ and $\setB \in \setF$, 
$\int_{\setA}\pi(\md x){P}(x,\setB)
=
\int_{\setB}\pi(\md x) {P}(x,\setA).
$
This article utilises a number of existing results for reversible Markov chains on the relationship between variance bounding, conductance, spectral gaps and geometric ergodicity. Here we define conductance and spectral gaps and summarise the relationships between the four properties.

Define the Hilbert space $L_0^2(\pi)=\{f:\mathcal{X}\rightarrow \mathbb{R}; \Expects{\pi}{f(X)}=0,~\Expects{\pi}{f^2(X)}<\infty\}$ with the inner product $\langle f, g \rangle=\int_{\mathcal{X}}f(x)g(x)\pi(\md x)$, and consider $P$ as an operator acting on $L_0^2(\pi)$ according to $(Pf)(x)=\int_{\mathcal{X}} P(x,\md y) f(y)$. If $P$ is reversible then it is a self adjoint operator on $L_0^2(\pi)$ and by the spectral theorem for bounded self-adjoint operators, for each $f\in L_0^2(\pi)$, $\langle f,P^nf\rangle =\int_{-1}^1 \lambda^nH_{f}(\md \lambda)$ for some positive measure $H_f$ on $[-1,1]$. Let
\begin{equation}
  \label{eqn.var.rep.of.gaps}
r_1:=\inf_{f\in L_0^2(\pi)}\frac{\langle f,Pf\rangle}{\langle f,f\rangle}\ge -1
~~~\mbox{and}~~~
r_2:=\sup_{f\in L_0^2(\pi)}\frac{\langle f,Pf\rangle}{\langle f,f\rangle}\le 1,
\end{equation}
or, equivalently \cite[e.g.][p320, Theorem 2]{Yosida1980}, that the smallest closed interval containing the support of $H_f$ for all $f\in L_0^2(\pi)$ is $[r_1,r_2]$. The \emph{spectral gap} of $P$ is $1-\max(|r_1|,|r_2|)$ \cite[e.g.][]{Geyer1992,RnR1997}, the \emph{right spectral gap} is $1-r_2$ and the \emph{left spectral gap} is $1+r_1$. $P$ is said to have a spectral gap (or a left or right spectral gap) if its spectral gap (or left or right spectral gap) is non-zero.

For any set $\setA\in \setF$ with $\pi(\setA)>0$ consider the probability of leaving 
$\setA$ at the next iteration given that the stationary chain is currently in $\setA$:
\[
\kappa(\setA):=
\frac{1}{\pi(\setA)}
\int_{\setA}P(x,\setA^c)\pi(\md x).
\]
The \emph{conductance}, $\kappa$ for a Markov kernel $P$ with invariant measure
$\pi$ is then \cite[e.g.][]{LawlerSokal:1988} \cite[see also][]{JerrumSinclair1988}
\[
\kappa:=\inf_{\setA:0<\pi(\setA)\le 1/2} \kappa(\setA).
\]

For any reversible Markov chain we have the following relationships:
\begin{align*}
  P\mbox{ is variance bounding}
  &\Leftrightarrow
  P\mbox{ has a right spectral gap}~~~~~~~\mbox{(Thrm 14 of \cite{RobRos:2008})}\\
  &\Leftrightarrow
  P\mbox{ has a positive conductance}~~~\mbox{(Thrm 2.1 of \cite{LawlerSokal:1988})}\\
  &\Leftarrow
  P\mbox{ has a spectral gap}\\
  &\Leftrightarrow
  P\mbox{ is geometrically ergodic}~~~~~~~\mbox{(Thrm 2.1 of \cite{RnR1997})}.
  \end{align*}
These relationships will be used repeatedly in the sequel without further reference.
 




\subsection{Example algorithms}
\label{sect.algorithms}
To exemplify our theoretical results we will consider four specific, frequently-used
MH algorithms.
\begin{enumerate}
\item The Metropolis-Hastings independence sampler (MHIS): $q(x,y)=q(y)$.
\item The random walk Metropolis (RWM): $q(x,y)=q(x-y)=q(y-x)$; \emph{e.g.}
  $q(x,y)=\mathsf{N}(y;x,\lambda^2 I)$.
\item The Metropolis-adjusted Langevin algorithm (MALA): 
$q(x,y)=\mathsf{N}(y;x+\frac{1}{2}\lambda^2\nabla \log \pi,\lambda^2 I)$.
\item The truncated MALA: 
\begin{equation}
\label{eqn.TMALA.prop}
q(x,y)=\mathsf{N}\left(y;x+\frac{1}{2}\lambda^2R(x),\lambda^2I\right),
~\mbox{where}~ 
R(x)=\frac{D\nabla \log \pi}{D\vee \Norm{\nabla \log   \pi}},
\end{equation}
for some $D>0$.
\end{enumerate}
In Proposals of type 2, 3 and 4, $\lambda$ is often referred to as the scale parameter of the proposal. 
The MHIS and RWM have been used since the early days of MCMC
\cite[e.g.][]{Tierney:1994}; conditions under which they
are geometrically ergodic (and, hence, variance bounding) have been well studied; see, for example,
\cite{Liu:1996} and \cite{MengTweed:1996} for the MHIS and
\cite{MengTweed:1996}, \cite{RobTweed1996RWM} and \cite{JarnHans2000}
for the RWM. Essentially, for the MHIS the proposal, $q$, must not
have lighter tails than the target, and for the RWM the target must
have suffiently smooth and exponentially decreasing tails. The MALA
was introduced in \cite{Besag1994} and was analysed in
\cite{RobTweed1996MALA}, in which the truncated MALA was also
introduced. The MALA can be much
more efficient than the RWM in moderate to high dimensions. As with
the RWM, for geometric ergodicity the MALA 
requires exponentially decreasing tails, but if the tails decrease too
quickly, $\Norm{\nabla \log \pi}$ grows too quickly and 
the MALA can fail to be geometrically ergodic. The truncated MALA
 circumvents this problem. 

In \cite{BGLR:2015} it is shown that the geometric ergodicity of an RWM algorithm need not be inherited
by the resulting DA algorithm. 

\begin{example}
\label{example.BGLR}
\cite[]{BGLR:2015}
Let $\setX=\mathbb{R}$ with $\pi(x)\propto e^{-x^2/2}$ and
$q(x,y)\propto e^{-(y-x)^2/(2\lambda^2)}$. If $\pihat(x)\propto
e^{-x^2/(2\sigma^2)}$, with $\sigma^2<1$ then $\PDA$ is not
geometrically ergodic. 
\end{example} 
The following conditional equivalence (proved in Section \ref{sec.proof.background}) is used throughout the sequel. If the parent kernel is geometrically ergodic then the DA kernel must have a left spectral gap, and with this constraint geometric ergodicity and variance bounding are equivalent.

\begin{proposition}
\label{prop.GEtonoGEbad}
Let $\PMH$ be a MH kernel targeting $\pi$ as specified in
\eqref{eqn.PMH}. Let $\PDA$ be the DA kernel derived from this
through the approximation $\pihat$ as in \eqref{eqn.PDA}. If $\PMH$
is geometrically ergodic then
$$
\PDA\mbox{ is geometrically ergodic }\iff~\PDA\mbox{ is variance bounding.}
$$
\end{proposition}

The original random walk Metropolis algorithm on $\pi(x)$ \textit{is} geometrically
ergodic \cite[]{MengTweed:1996}, and hence variance bounding, so the DA kernel in Example \ref{example.BGLR} has not inherited its parent's desirable properties. As a 
direct corollary of our Theorem
\ref{thrm.heavy} (see Section \ref{sec.DAwithSame}) we find that $\sigma^2\ge 1$ is exactly the right condition in this case:
\begin{example}
\label{example.BGLR.counter}
Let $\setX=\mathbb{R}$ with $\pi(x)\propto e^{-x^2/2}$ and
$q(x,y)\propto e^{-(y-x)^2/(2\lambda^2)}$. If 
$\pihat(x)\propto e^{-x^2/(2\sigma^2)}$, with $\sigma^2\ge 1$ then $\PDA$ is 
variance bounding and geometrically ergodic. 
\end{example} 

Examples \ref{example.BGLR} and \ref{example.BGLR.counter} suggest an
 intuition that problems may arise when $\pihat(x)$
has lighter tails than $\pi(x)$. As we shall see, this is a part of
the story; however, in general, heavier tails are not sufficient to
guarantee inheritance of the variance bounding property, and for a class
of algorithms where heavy tails are sufficient, lighter tails can also
be sufficient provided they are not too much lighter, in a sense we
make precise. 


\section{Variance bounding: inheritance and equivalence}
\label{sect.generic}
Throughout this section we use the following generic formulation for two Markov kernels.

\begin{definition}
  \label{defn.two.generic.kernels}
Let $P_A(x,\md y)$ and $P_B(x,\md y)$ be propose-accept-reject Markov
kernels both targeting a distribution $\pi$, and
using, respectively, 
proposal densities of $q_A(x,y)$ and $q_B(x,y)$ and  
acceptance probabilities of $\alpha_A(x,y)$ and
$\alpha_B(x,y)$.
\end{definition}

Theorem \ref{thrm.main}, below, follows from Lemma \ref{lemma.general}, which is proved in Section \ref{sect.prove.main}. It
generalises Corollary 12 of \cite{RobRos:2008} to allow for different
acceptance probabilities and, more importantly, removes the need for a
fixed, uniform minorisation condition. The minorisation needs only hold in a region $\mathcal{D}(x)$ such that under $P_A$ there is ``unlikely'' to be an accepted proposal in $\mathcal{D}(x)^{\complement}$.

\begin{lemma}
\label{lemma.general}
Let $P_A$, $P_B$, $q_A$, $q_B$, $\alpha_A$ and $\alpha_B$ be as in Definition \ref{defn.two.generic.kernels}, and let the conductances of $P_A$ and $P_B$ be $\kappa_A$ and $\kappa_B$ respectively. 
If $\kappa_A>0$ and there is an $\epsilon<\kappa_A$ and
a $\delta>0$ 
such that for $\pi$-almost all
$x\in\setX$, there is a
region $\setD(x)\in \setF$ such that
\begin{equation}
\label{eqn.genlocal}
\int_{\setD(x)^{\complement}}q_A(x,y)\alpha_A(x,y)\md y\le \epsilon,
\end{equation}
and
\begin{equation}
\label{eqn.genabscts}
y\in \setD(x)\Rightarrow q_B(x,y)\alpha_B(x,y)~\ge~ \delta~ q_A(x,y)\alpha_A(x,y),
\end{equation}
then 
 $\kappa_B \ge (1-\epsilon/\kappa_A) \delta \kappa_A$.
\end{lemma}

If $P_A$ is variance bounding, $\kappa_A>0$; 
 choose an $\epsilon \in (0,\kappa_A)$ and for
each $x$ a corresponding $\setD(x)$ so as to satisfy 
\eqref{eqn.genlocal} and \eqref{eqn.genabscts} to obtain:

\begin{theorem}
\label{thrm.main}
Let $P_A$, $P_B$, $q_A$, $q_B$, $\alpha_A$ and $\alpha_B$ be as in Definition \ref{defn.two.generic.kernels}. If $P_A$ is variance bounding and for any $\epsilon>0$
there is a $\delta>0$ such that for $\pi$-almost all $x\in\setX$ there
is a region $\setD(x)\in \setF$ such that 
\eqref{eqn.genlocal} and
\eqref{eqn.genabscts} hold,  then $P_B$ is also variance bounding.
\end{theorem}

The
relationship between conductance and right spectral gap has recently \cite[]{LeeLatus2014,RudolphSprungk2016} been used in other
contexts to bound the behaviour of one Markov kernel in terms of that
of another.
Lemma \ref{lemma.general} itself shows that condition \eqref{eqn.genlocal} 
  need only hold for a single $\epsilon<\kappa_{A}$; however, since in practice  
$\kappa_{A}$ is
  unlikely to be known, the conditions of Theorem \ref{thrm.main} are more practically
  useful. 

From Section \ref{sect.results} we apply Theorem \ref{thrm.main} to provide sufficient conditions
for a delayed-acceptance kernel to inherit variance bounding from its
Metropolis-Hastings parent. However, if a DA kernel is variance
bounding then so is its parent MH kernel. Thus, the sufficient conditions
in Section \ref{sect.results} imply an \textit{equivalence} between the two kernels
with respect to the variance bounding property. In this section, after
two key definitions, we return, briefly, to this equivalence with
regard to the variance bounding property and
provide sufficient conditions for equivalence (over potential targets) between 
Metropolis-Hastings kernels arising from two different proposal densities.  

The most natural special case of \eqref{eqn.genlocal} in practice is
where the kernel is \textit{uniformly local}, which we define as follows: 

\begin{definition}
\label{def.local}
(\textbf{Uniformly Local}) A proposal is \textit{uniformly local} if, given any 
$\epsilon>0$,
\begin{equation}
\label{eqn.localA}
\exists ~r<\infty~ \mbox{such that}~\mbox{for all}~x\in\mathcal{X},~~~\int_{B(x,r)^c}q(x,y)\md y <
\epsilon.
\end{equation} 
A propose-accept-reject kernel is defined to
be uniformly local when its proposal is uniformly local.
\end{definition}

Here and throughout this article, $B(x,r):=\{y\in\setX:\Norm{y-x}<r\}$ is the open ball of radius
$r$  centred on $x$. In our examples, $\Norm{x}$
  indicates the Euclidean norm, although the results are equally valid
  for other norms such as the Mahalanobis norm.

Control of the ratio
$q(y,x)/q(x,y)$ will also be important and so we define the following.

\begin{definition}
\label{eqn.define.deltaq}
For any proposal density $q(x,y)$,
\[
\Delta(x,y;q):=\log q(y,x) - \log q(x,y).
\]
\end{definition}

Clearly, the RWM is a uniformly local kernel; moreover $\Delta(x,y;q_{\rm RWM})=0$. In
contrast, on any target with unbounded support, the MHIS
cannot be uniformly local; as we shall see, the behaviour of $\Delta$ is
then irrelevant. For the MALA and the
truncated MALA we have:
\begin{proposition}
\label{prop.TMALA.conditions}
\textcolor{white}{.}\\(A) Let $q(x,y)$ be the proposal for the truncated MALA in
\eqref{eqn.TMALA.prop} or for the MALA on a target where 
$\mbox{ess sup}_x \Norm{\nabla \log \pi(x)}= D<\infty$. Then\\ 
(i) For all $x$, $\Probs{q}{\Norm{Y-x}>r}\rightarrow 0$ uniformly in
$x$ as $r\rightarrow \infty$,
so $q$ is \textit{uniformly local}, as defined in \eqref{eqn.localA}.\\
(ii) 
$|\Delta(x,y;q)|\le h(\Norm{y-x}):=D\Norm{y-x} +\lambda^2D^2/8$.\\
(B) The proposal, $q(x,y)$, for the MALA on a target where 
$\mbox{ess sup} \Norm{\nabla \log \pi(x)}= \infty$ 
is not \textit{uniformly local}. 
\end{proposition}

The applicability of Lemma \ref{lemma.general} and Theorem
\ref{thrm.main} ranges beyond delayed-acceptance kernels. Here we
supply sufficient conditions for an \textit{equivalence} between
Metropolis--Hastings proposals.

\begin{theorem}
\label{thrm.sideways}
Let $P_A$, $P_B$, $q_A$, $q_B$, $\alpha_A$ and $\alpha_B$ be as in Definition \ref{defn.two.generic.kernels} except that $q_A(x,y)$ and $q_B(x,y)$ are \textit{uniformly local} proposal kernels, with 
$\log q_A(x,y) - \log q_B(x,y)$ a continuous function from
$\mathbb{R}^{2d}$ to $\mathbb{R}$.  
If, for $\pi$-almost all $x$ and for some function 
$h:[0,\infty)\rightarrow [0,\infty)$ with
    $h(r)<\infty$ for all $r<\infty$,
\begin{equation}
\label{eqn.unifdiscrepMH}
|\Delta(x,y;q_B)-\Delta(x,y;q_A)|\le h(\Norm{y-x}),
\end{equation}
then $P_A$ is variance bounding if and
only if $P_B$ is variance bounding. 
\end{theorem}

Thus, for example, any two random-walk Metropolis algorithms with
Gaussian jumps are equivalent, in that if, on a particular target, one
is variance bounding then so is the other. When restricted to targets
with a continuous gradient this equivalence extends to
truncated MALA algorithms. The continuity requirement on 
$\log q_A-\log q_B$ rules out, for example, an equivalence between a
Gaussian random walk and a random walk where the proposal has bounded
support; indeed, the latter may not even be ergodic if the target has
gaps in its support.


\section{Application to delayed-acceptance kernels}
\label{sect.results}
\subsection{Key definitions and properties}
For uniformly local kernels we will describe two general sets of sufficient
conditions for \eqref{eqn.genabscts} to hold. The first is based upon
the fact that the acceptance probability for $\PDA$ can be written as
\begin{align*}
\alphaDA(x,y)
&=
[1\wedge
  \rDA_1(x,y)]~\left[1\wedge\frac{\rMH(x,y)}{\rDA_1(x,y)}\right]
\ge
[1\wedge
  \rDA_1(x,y)]~\left[1\wedge\frac{1}{\rDA_1(x,y)}\right]~\left[1\wedge\rMH(x,y)\right],\\
\mbox{or}~~~\alphaDA(x,y)
&=
\left[1\wedge\frac{\rMH(x,y)}{\rDA_2(x,y)}\right]~[1\wedge
  \rDA_2(x,y)]
\ge
\left[1\wedge{\rMH(x,y)}\right]
~\left[1\wedge\frac{1}{\rDA_2(x,y)}\right]
~[1\wedge  \rDA_2(x,y)],
\end{align*}
where $\rDA_1$ and $\rDA_2$ are as defined in \eqref{eqn.define.rrr}.
So, if $|\log \rDA_1(x,y)|\le m$ or $|\log \rDA_2(x,y)|\le m$ then $\alphaDA(x,y)\ge
e^{-m}\alphaMH(x,y)$. The quantity 
$|\log \rDA_2(x,y)|=|[\log \pihat(y)-\log \pi(y)]-[\log
  \pihat(x)-\log \pi(x)]|$ 
measures
the discrepancy between the error in the approximation at the proposed
value and the error in the approximation at the current value. We name
this intuitive 
quantity, the \textit{log-error discrepancy}. The quantity $\log \rDA_1$
is less natural since it relates $\pihat(x), \pihat(y)$
and $q(x,y)$.

The second set of conditions is based upon the fact that if either
$\rDA_1(x,y)\le 1$ and $\rDA_2(x,y)\le 1$ or if $\rDA_1(x,y)\ge 1$ and
$\rDA_2(x,y)\ge 1$ then $\alphaDA(x,y)=\alphaMH(x,y)$, whatever the
log-error discrepancy.

These considerations lead to the natural definitions of a `potential problem' set, $\setM_m(x)$, and
a `no problem' set $\setC(x)$, as follows:
\begin{eqnarray}
\label{eqn.setM}
\setM_m(x)&:=&\{y\in\setX: \min|\log \rDA_1(x,y))|,|\log
  \rDA_2(x,y)|>m\},\\
\label{eqn.setC}
\setC(x)&:=&\{y\in\setX:\mbox{sign}[\log \rDA_1(x,y)]~\mbox{sign}[\log
  \rDA_2(x,y)]\ge 0\}.
\end{eqnarray}
Theorem \ref{thrm.main} then leads directly to the following.

\begin{corollary}
\label{cor.gen.da}
Let $\PMH$ be the Metropolis-Hastings kernel given in
\eqref{eqn.PMH} and let $\PDA$ be the corresponding delayed-acceptance
kernel given in \eqref{eqn.PDA}. 
Suppose that for all   $\epsilon>0$ 
there is an $m<\infty$ such that for $\pi$-almost all $x$ there exists a set $\setD(x)\subseteq\mathcal{X}$ such that
\begin{equation}
\label{eqn.nice.inside}
\setM_m(x)\cap\setD(x)\subseteq \setC(x),
\end{equation}
and
\begin{equation}
\label{eqn.small.outside}
\int_{\setD(x)^c}q(x,y)\md y \le \epsilon.
\end{equation}
Subject to these conditions, if $\PMH$ is variance bounding then so
is $\PDA$.
\end{corollary}

When $\pihat$ has heavier tails than $\pi$ then for large $x$, the set
$\setC(x)$ can play an important role in the inheritance of the
variance bounding property. 
In a dimension $d>1$, there are numerous possible definitions
of `heavier tails'. The following is precisely that required for our
purposes:
\begin{definition}
  \label{def.heavy.tail}
(\textbf{heavy tails}) An approximate density $\pihat$ is said to
  have \textit{heavy tails} with respect to a density $\pi$ if
\begin{equation}
\label{eqn.heavy}
\exists~ r_*>0~\mbox{such that if}~\Norm{x}>r_*~\mbox{and}~\Norm{y}>r_*~\mbox{then}~\pihat(x)\le\pihat(y)
\Rightarrow \frac{\pihat(x)}{\pi(x)}\ge\frac{\pihat(y)}{\pi(y)}.
\end{equation}
\end{definition}
Intuitively, the left hand side is true when $x$ is `further from the
centre' 
 (according to $\pihat$) than
$y$, and the implication is that the further out a point, the larger
$\pihat$ is compared with $\pi$.

For uniformly local kernels
 we show (Corollary \ref{cor.growdiscrep}) that it is sufficient that either
the log error discrepancy should satisfy a  growth condition that is
uniform in $||y-x||$, or (Theorem \ref{thrm.heavy})
that the
tails of the approximation should be heavier than those of the target
and that $|\Delta(x,y;q)|$ should satisfy a
growth condition that is uniform in $\Norm{y-x}$. 

For all kernels, boundedness of the error $\pihat(x)/\pi(x)$ away
from $0$ and $\infty$ 
will ensure the required inheritance (Corollary \ref{cor.bounded}). This is a very strong
condition, but we exhibit MHIS and MALA algorithms where the weaker
conditions, that are sufficient for a uniformly local kernel, are satisfied,
but the DA kernel is not variance bounding even though the MH kernel is.

\subsection{DA kernels with the same proposal
  distribution as the parent}
\label{sec.DAwithSame}
Suppose that for all $x\in\setX$, 
$\gamma_{lo}\le \pihat(x)/\pi(x)\le \gamma_{hi}$, then
$|\log \rDA_2(x,y)|\le \log(\gamma_{hi}/\gamma_{lo})$, so applying Corollary \ref{cor.gen.da} with $\setD(x)=\mathcal{X}$ and
$m=\log\gamma_{hi}-\log \gamma_{lo}$ leads to:
\begin{corollary}
\label{cor.bounded}
Let $\PMH$ and $\PDA$ be as described in Corollary 
\ref{cor.gen.da}. If there
exist 
$\gamma_{lo}>0$ and $\gamma_{hi}<\infty$ such that
$\gamma_{lo}\le\pihat(x)/\pi(x)\le \gamma_{hi}$, and if $\PMH$ is variance bounding
then so is $\PDA$.
\end{corollary}
A more direct proof of Corollary \ref{cor.bounded} is possible using Dirichlet forms. However, Corollary \ref{cor.gen.da} comes into its
own when the error discrepancy is unbounded. 

We
first provide a cautionary example which shows that once the errors
are unbounded the delayed-acceptance kernel need not inherit
the variance bounding property from the Metropolis-Hastings kernel even if
the growth of the log error discrepancy is uniformly bounded or if $\pihat$ has heavier tails than $\pi$.

\begin{example}
\label{example.MHIS}
Let $\setX=\mathbb{R}$, let $\PMH$ be an MHIS with
$q(x,y)=q(y)=\pi(y)=e^{-y}\mathds{1}(y>0)$, and let $\PDA$ be the
corresponding delayed-acceptance kernel \eqref{eqn.PDA}, with 
 $\pihat(y)=ke^{-ky}\mathds{1}(y>0)$ with $k>0$ and $k \ne 1$. 
 $\PMH$ is geometrically ergodic,
but $\PDA$ is neither geometrically ergodic nor variance bounding.
\end{example} 

The problem with the algorithm in Example \ref{example.MHIS} is that
for some $x$ values the proposal, $y$, is very likely to be a long way from
$x$ and yet $y\notin\setC(x)$. Our definition of a \textit{uniformly local}
proposal, \eqref{eqn.localA}, provides uniform control on the
probability that $\Norm{y-x}$ is large. Since this is only strictly
necessary for $y\notin \setC(x)$, \eqref{eqn.localA} is stronger
than necessary, but it is much easier to check.

Our first sufficient condition for uniformly local kernels insists on
uniformly bounded growth in the log-error discrepancy except when
$\alphaDA(x,y)=\alphaMH(x,y)$. For $\pi$-almost all $x$ and for some function 
$h:[0,\infty)\rightarrow [0,\infty)$ with
    $h(r)<\infty$ for all $r<\infty$,
\begin{equation}
\label{eqn.unifdiscrep}
\{y\in\setX:|\log \rDA_2(x,y)|>h(\Norm{y-x})\}\subseteq \setC(x).
\end{equation}

If a proposal is uniformly local, given $\epsilon>0$ find $r(\epsilon)$ according to
\eqref{eqn.localA}. Then \eqref{eqn.unifdiscrep} implies that for 
$y\in B(x,r)$, $\setM_{h(r)}\subseteq \setC(x)$. Applying Corollary
\ref{cor.gen.da} with $\setD(x)=B(x,r)$ leads to the following.

\begin{corollary}
\label{cor.growdiscrep}
Let $\PMH$ and $\PDA$ be as described in Corollary
\ref{cor.gen.da}. In addition let $q(x,y)$ be a uniformly local proposal as in  
\eqref{eqn.localA}, and
let the error discrepancy satisfy \eqref{eqn.unifdiscrep}. If
$\PMH$ is variance bounding then so is $\PDA$.
\end{corollary}

Because most of the mass from the proposal, $y$, is
not too far away from the current value, $x$, the discrepancy between the error at
$x$ and the error at $y$ remains manageable provided the discrepancy grows
in a manner that is controlled uniformly across the statespace. Since the random walk Metropolis on an
exponential target density is geometrically ergodic
\cite[]{MengTweed:1996} we may apply Corollary \ref{cor.growdiscrep} with
$h(r)=|k-1|r$, and then Proposition \ref{prop.GEtonoGEbad}, to obtain the following contrast to Example
\ref{example.MHIS},
 and showing that the variance bounding property can be inherited even when the
approximation has \textit{lighter} tails than the target.

\begin{example}
\label{example.RWMa}
Let $\setX=\mathbb{R}$ and let $\PMH$ be a RWM algorithm on $\pi(x)=e^{-x}\mathds{1}(x>0)$ using
$q(x,y)\propto e^{-(y-x)^2/(2\lambda^2)}$. For any $k>0$, let $\PDA$
be the corresponding delayed-acceptance RWM algorithm 
 using a surrogate of
$\pihat(x)=ke^{-kx}\mathds{1}(x>0)$. $\PDA$ is variance bounding and geometrically ergodic.
\end{example}



As yet, the set $\setC(x)$ has not played a part in any of our examples. It is precisely this
set that allows a delayed-acceptance random walk Metropolis
kernel to inherit the variance bounding property 
 from its parent even when the error
discrepancy is not controlled uniformly, provided  $\pihat$
has tails that are heavier than those of $\pi$.
For general MH algorithms an additional control on the behaviour of
$q$ is enough to guarantee
inheritance of the variance bounding property.

\begin{theorem}
\label{thrm.heavy}
Let $\PMH$ be the Metropolis-Hastings kernel given in
\eqref{eqn.PMH} and let $\PDA$ be the corresponding delayed-acceptance
kernel given in \eqref{eqn.PDA}. 
Further, let $q(x,y)$ be a uniformly local proposal in
the sense of 
\eqref{eqn.localA}, let $\pi$ and $\pihat$ be continuous, and
let $\pihat$ have heavier tails than $\pi$ in the sense of
\eqref{eqn.heavy}. Suppose that, in addition, for any $\setD(x)$ required by
\eqref{eqn.nice.inside} and \eqref{eqn.small.outside} there exists a
function $h:[0,\infty)\rightarrow [0,\infty)$ with
    $h(r)<\infty$ for all $r<\infty$, such that for $\pi$-almost all
    $x$
\begin{equation}
\label{eqn.linear.q}
\{y\in\setD(x):~|\Delta(x,y;q)|>
h(\Norm{y-x})\}\subseteq\setC(x).
\end{equation}
Subject to these conditions, if $\PMH$ is variance bounding then so
is $\PDA$.
\end{theorem}

We now consider the delayed-acceptance versions of the random walk Metropolis, 
 the truncated MALA, and the MALA. Before doing this we
provide the details of a property that was anticipated in
\cite{RobTweed1996MALA}.

\begin{proposition}
\label{prop.TMALA.GE}
Let $\PMH_{\rm RWM}$ be a random walk Metropolis kernel using $q(x,y)\propto
e^{-\frac{1}{2\lambda^2}\Norm{y-x}^2}$ and 
targeting a density $\pi(x)$. Let $\PMH$ be a
Metropolis-Hastings kernel on $\pi$ of the form
$q(x,y)\propto e^{-\frac{1}{2}\lambda^2\Norm{y-x-v(x)}^2}$,
where $\mbox{$\pi$-ess sup}_x \Norm{v(x)}<\infty$.  $\PMH_{\rm RWM}$ is
variance bounding if and only if $\PMH$ is variance bounding.
\end{proposition}

 Proposition \ref{prop.TMALA.GE} clearly applies to 
a truncated MALA kernel on $\pi(x)$ using
$q$ as in \eqref{eqn.TMALA.prop}. It, together with each of our subsequent results for the
 truncated MALA, 
also applies to a MALA kernel on a target where 
$\mbox{$\pi$-ess sup}_x \Norm{\nabla   \log \pi(x)}= D<\infty$; in practice, however, the useful
set of such kernels is limited to targets with exponentially decaying
tails, since MALA is not geometrically ergodic on targets with heavier
tails \cite[]{RobTweed1996MALA}.

Given Proposition \ref{prop.TMALA.conditions} and its prelude,  a direct application of Theorem \ref{thrm.heavy} then leads to the following.

\begin{example}
\label{example.RWMTMALA}
Let $\PMH_{\rm RWM}$ and $\PMH_{\rm TMALA}$ be, respectively, a random walk
Metropolis kernel and a truncated MALA kernel
 on the differentiable density, $\pi(x)$. Let $\PDA_{\rm RWM}$ and
 $\PDA_{\rm TMALA}$ be the corresponding 
delayed-acceptance kernels, created as in
\eqref{eqn.PDA} through the
continuous density, $\pihat(x)$. Suppose also that $\pihat$ has heavier tails than
$\pi$ in the sense of \eqref{eqn.heavy}. Subject to these conditions,
if $\PMH_{\rm RWM}$ is variance bounding then so is $\PDA_{\rm RWM}$, and
if $\PMH_{\rm TMALA}$ is variance bounding then so is 
$\PDA_{\rm TMALA}$.
\end{example}

The MALA is geometrically ergodic when applied to
one-dimensional targets of the form 
$\pi(x)\propto e^{-|{x}|^\beta}$ for $\beta \in [1,2)$
  \cite[]{RobTweed1996MALA}; when $\beta=2$ geometric ergodicity
  occurs provided $\lambda$ is sufficiently small, and for $\beta>2$
  the MALA is not geometrically ergodic. 
Even when $\beta >1$, however,
  Theorem \ref{thrm.heavy} does not apply because the proposal
  is not uniformly local.  

\begin{example}
\label{example.MALAa}
Let $\setX=\mathbb{R}$ and let $\PMH$ be a MALA algorithm on $\pi(x)\propto e^{-x^\beta}1(x>0)$
with $1\le \beta<2$. Let
$\pihat(x)\propto e^{-x^\gamma}1(x>0)$
 and let $\PDA$ be the corresponding delayed-acceptance MALA kernel
\eqref{eqn.PDA} (i.e. using  
a proposal of $Y=x+\frac{1}{2}\lambda^2\nabla \log \pi(x)+\lambda Z$,
where $Z\sim N(0,1)$).  $\PDA$ is
neither geometrically ergodic nor variance bounding, except when 
 $\gamma=\beta$.
\end{example}

The contrast between the truncated MALA and the MALA in Examples
\eqref{example.RWMTMALA} and \eqref{example.MALAa} highlights the
importance of a uniformly local proposal. In practice, however, if $\pi(x)$ is computationally expensive to evaluate then,
typically, $\nabla \log \pi(x)$ will also be expensive to evaluate and
it might seem more reasonable to base the proposal for
delayed-acceptance 
MALA and
delayed-acceptance truncated MALA on 
$\nabla \log \pihat(x)$.

\subsection{Kernels where the proposal is based upon $\pihat$}
\label{sect.resultsB} 
On some occasions, the proposal $q(x,y)$ is a function of the
posterior, $\pi(x)$, and on such occasions it may be expedient for the
delayed-acceptance algorithm to use a proposal
$\qhat(x,y)$, which is based upon $\pihat(x)$. The acceptance rate is  
 $\alphaDA_b(x,y)=[1\wedge \rDA_{1b}(x,y)][1\wedge \rDA_2(x,y)]$, where
\[
\rDA_{1b}(x,y):=\frac{\pihat(y)\qhat(y,x)}{\pihat(x)\qhat(x,y)}.
\]
With $\alphabarDA_{b}(x):=\Expects{q}{\alphaDA_{b}(x,Y)}$, the
corresponding delayed acceptance kernel is
\begin{equation}
\label{eqn.PDAb}
\PDA_{b}(x,dy):=\qhat(x,y)\md y ~\alphaDA_{b}(x,y)+[1-\alphabarDA_{b}(x)]\delta_{x}(dy).
\end{equation}

Let $\rMH_{\rm hyp}(x,y):=\pi(y)\qhat(y,x)/[\pi(x)\qhat(x,y)]$, 
$\alphaMH_{\rm hyp}(x,y)=1\wedge \rMH_{\rm hyp}(x,y)$, and, 
with $\alphabarMH_{\rm hyp}(x)=\Expects{\qhat}{\alphaMH_{\rm hyp}(x,Y)}$, 
consider the hypothetical Metropolis-Hastings kernel:
\begin{equation}
\label{eqn.Phyp}
\PMH_{\rm hyp}(x,\md y):=\qhat(x,y)\md y~ \alphaMH_{\rm hyp}(x,y)+[1-\alphabarMH_{\rm hyp}(x))]\delta_{x}(dy).
\end{equation}
Now,
 $\alphaDA_b(x,y)\le \alphaMH_{\rm hyp}(x,y)$, so  if $\PMH_{\rm hyp}$ is not
variance bounding then $\PDA_b$ is not variance bounding
either. There is an exact correspondence between $\PMH$ from the
previous section, and $\PMH_{\rm hyp}$, and it is natural to consider inheritance of geometric
ergodicity from $\PMH_{\rm hyp}$ exactly as in the prevous section we
considered inheritance from $\PMH$. The theoretical results are
analogous and will not be restated; moreover, the theoretical
properties of kernels of the form $\PMH_{\rm hyp}$ are less well
investigated. Instead we illustrate inheritance of variance
bounding (or its lack) through two examples. 

\begin{example}
\label{ex.TMALAb}
Let $\PMH_{\rm TMALA}$ be, a truncated MALA kernel
 on the differentiable density, $\pi(x)$. Let 
 $\PDA_{\rm TMALAb}$ be the corresponding 
delayed-acceptance kernel, created as in
\eqref{eqn.PDAb} through the
differentiable density $\pihat(x)$. 
$\PDA_{\rm TMALAb}$ inherits the variance bounding property 
from $\PMH_{\rm TMALA}$ if either of the following conditions
holds.
(i) There is uniformly bounded growth in the log error discrepancy, in
the sense of \eqref{eqn.unifdiscrep}, or 
(ii) $\pihat$ has heavier tails than
$\pi$ in the sense of \eqref{eqn.heavy}.
\end{example}

Our penultimate example 
 suggests that a delayed-acceptance MALA based
upon an approximation that has heavier (though not too much heavier)
 tails is a reasonable choice. 

\begin{example}
\label{ex.MALAb}
Let $\setX=\mathbb{R}$ and let $\PMH$ be a MALA algorithm on $\pi(x)\propto e^{-x^\beta}\mathds{1}(x>0)$
with $1\le \beta<2$. Let
$\pihat(x)\propto e^{-x^\gamma}\mathds{1}(x>0)$
 and let $\PDA$ be the corresponding delayed-acceptance MALA kernel
created as in
\eqref{eqn.PDAb} through the
differentiable density $\pihat(x)$.
$\PDA$ is variance bounding $\iff$ $\PDA$ is geometrically ergodic $\iff 1\le \gamma\le \beta$.
\end{example}

We summarise the consequences of Examples \ref{example.RWMTMALA} to \ref{ex.MALAb} for $\pi(x)\propto e^{-x^\beta}\mathds{1}(x>0)$ and $\pihat(x)\propto e^{-x^\gamma}\mathds{1}(x>0)$ in Table \ref{table.summary}, filling in the two blanks with Example \ref{ex.TMALA.not.GE} below. The table displays the results in terms of variance bounding, which is equivalent to geometric ergodicity in all these cases by Proposition \ref{prop.GEtonoGEbad}.

\begin{example}
  \label{ex.TMALA.not.GE}
Let $\setX=\mathbb{R}$ and let $\PMH$ be a RWM or truncated MALA algorithm on $\pi(x)\propto e^{-x^\beta}\mathds{1}(x>0)$
with $1\le \beta<2$. Let
$\pihat(x)\propto e^{-x^\gamma}\mathds{1}(x>0)$
 and let $\PDA$ be the corresponding delayed-acceptance RWM or truncated MALA kernel
created either from 
\eqref{eqn.PDA} or \eqref{eqn.PDAb} through the
differentiable density $\pihat(x)$. If $1<\beta<\gamma$, 
$\PDA$ is neither geometrically ergodic nor variance bounding.
\end{example}

\begin{table}
\begin{tabular}{l|cccc}
  Algorithm & MALA & RWM/TMALA & MALA ($\nabla \log \pihat$) & TMALA ($\nabla\log \pihat$)\\
  \hline
  $1\le \gamma<\beta\le 2$&$\times$ Ex \ref{example.MALAa}&\checkmark $ $ Ex \ref{example.RWMTMALA}&\checkmark $ $ Ex \ref{ex.MALAb} &\checkmark $ $ Ex \ref{ex.TMALAb}\\
  $1\le \beta<\gamma$&$\times$ Ex \ref{example.MALAa}&$\times$~Ex \ref{ex.TMALA.not.GE}&$\times$ Ex \ref{ex.MALAb} & $\times$~Ex \ref{ex.TMALA.not.GE}
\end{tabular}
\caption{Whether or not the DA algorithm for $\pi(x)\propto e^{-x^\beta}\mathds{1}(x>0)$ using $\pihat(x)\propto e^{-x^\gamma}\mathds{1}(x>0)$ is variance bounding as a function of $\gamma$ and $\beta$ and the specific DA algorithm. The final two columns indicate that $\pihat$ rather than $\pi$ is used to create the proposal. \label{table.summary}  }
\end{table}

\section{Numerical demonstrations}
\label{sec.numerical}
The theoretical results from Section \ref{sect.results} were made more concrete through Examples \ref{example.BGLR} to \ref{ex.TMALA.not.GE}. In this section we investigate the numerical performance of delayed acceptance algorithms in examples similar to those used in earlier sections. The specific targets in the earlier Examples were chosen to demonstrate particular points as simply as possible; here we deliberately investigate a broader class of targets, the exponential family class \cite[e.g.][]{RobTweed1996MALA,Livingstoneetal2019}:
\begin{equation}
  \label{eqn.WOSSNAME}
  \pi(x)\propto \exp\left(-||x||^\beta\right)
  ~~~\mbox{and}~~~
  \pihat(x)\propto \exp\left(-||x||^\gamma/\kappa^\gamma\right).
\end{equation}
The parameters $\beta$ and $\gamma$ in (19) govern the lightness of the tails in the target and the approximation to it respectively, and allow us to vary these separately.

A lack of variance bounding can be seen in terms of the chain struggling to leave a certain region, which typically has a low probability under $\pi$.
In practice, this lack of variance bounding (or a lack of geometric ergodicity) can manifest in two ways.
\begin{enumerate}
\item When a sensible starting value is not known, a starting value with poor properties may be chosen unwittingly and the algorithm may struggle to move from this initial point or region of the space.
  \item Even when started from a reasonable value, over the course of a sufficiently long run the algorithm will visit this ``danger region'' and then struggle to leave.
\end{enumerate}
For the target \eqref{eqn.WOSSNAME}, the ``danger region'' corresponds to the tails of $\pi$.

Our experiments deliberately start the algorithm in the tails of $\pi$ and then measure the number of iterations to reach the centre of the distribution. To make ``reaching the centre'' concrete, we find the number of iterations until $||x||$ is less than its median value under $\pi$. To decide where in the tails we start, we set $||x||$ to its $1-p_0$ quantile under $\pi$, for $p_0\in\{10^{-1},10^{-2},\dots,10^{-6}\}$ in Scenarios (i) and (ii), and $p_0\in\{10^{-4},10^{-8},\dots,10^{-24}\}$ in Scenarios (iii) and (iv); we start the algorithm from a uniformly random point on the surface of that hypersphere. In practical MCMC, many runs are of $\mathcal{O}(10^6)$ iterations, so it is not unreasonable that issues which are detected for $p_0\ge 10^{-8}$ might occur in practice even when the algorithm is started from a sensible value. We work in dimension $d=5$ and repeat each experiment $20$ times, except for scenario (iii) where we repeat $10$ times to avoid excessive clutter.

We consider four specific scenarios, and so as to bound the amount of computing time, in each scenario we set a maximum number of iterations for which the algorithm should be run. In all scenarios the time until convergence increases with the starting quantile, whether or not the algorithm is variance bounding, for the most part simply because the algorithm is starting further from the main mass of the target. However, when the disparity between algorithms grows towards an order of magnitude, this suggests danger.

For the DARWM and DAMALA, the scaling parameter, $\lambda$, was chosen so that for the RWM or MALA itself, the acceptance rate was a little larger than the theoretical optimum values of approximately $23\%$ and $57\%$ respectively. DATMALA used the same scaling as DAMALA and a truncation value such that when TMALA explored the true posterior, fewer than $4\%$ of the gradients were truncated.

\textbf{Scenario i} ($\beta=\gamma=2$, $\kappa=1/2$ and $\kappa=2$). The results appear in the top-left of Figure \ref{fig.experiments} and demonstrate the undesirable behaviour when the target and the approximation are both Gaussian but the approximation has lighter tails than the target (see Example \ref{example.BGLR}), and the reasonable behaviour when the approximation's tails are less tight than the target's (Example \ref{example.BGLR.counter}).

\textbf{Scenario ii} ($\beta=\gamma=1$, $\kappa=1/2$ and $\kappa=2$). The results, in the top-right of Figure \ref{fig.experiments}, demonstrate that, in alignment with Example \ref{example.MHIS}, the worst behaviour by some margin is exhibited by the only non-variance bounding algorithm: the independence sampler where $\pihat$ uses a smaller scaling than $\pi$ has. In particular, aligning with Example \ref{example.RWMa}, the DARWM that uses the same $\pihat$ as the poor independence sampler performs only marginally worse than the DARWM which uses the notionally `safer` $\pihat$.

\textbf{Scenario iii} ($\beta=1.5$, $\kappa=1$, $\gamma=1.2$ and $\gamma=1.8$). This corresponds to Examples \ref{example.RWMTMALA}, \ref{example.MALAa} and \ref{ex.TMALA.not.GE} and is consistent with the DARWM and DATMALA, but not DAMALA, being variance bounding when $1<\gamma<\beta$, and none being variance bounding when $\gamma>\beta$.

\textbf{Scenario iv} ($\beta=1.5$, $\kappa=1$, $\gamma=1.2$ and $\gamma=1.8$, proposal uses $\nabla \log \pihat$) and suggests that as with Examples \ref{ex.TMALAb} and  \ref{ex.MALAb} , DAMALA and DATMALA are both variance bounding when $\gamma<\beta$, and following Example \ref{ex.TMALA.not.GE}, neither is variance bounding when $\gamma >\beta$.

In scenarios (i), (iii) and (iv) the target itself has lighter-than exponential  tails, so even though the x-axis is linear in $\log p$ it is sublinear in the magnitude of the initial value, $||x_0||$. Hence, issues with the algorithms might be expected to appear more slowly as $-\log p_0$ increases than they do with scenario (ii). Whilst exceptionally poor behaviour is unlikely to be seen, therefore, during a typical run that has been started from the main posterior mass, it could easily occur as a result of a poor starting value.

\begin{figure}
\begin{center}
    \includegraphics[scale=0.45]{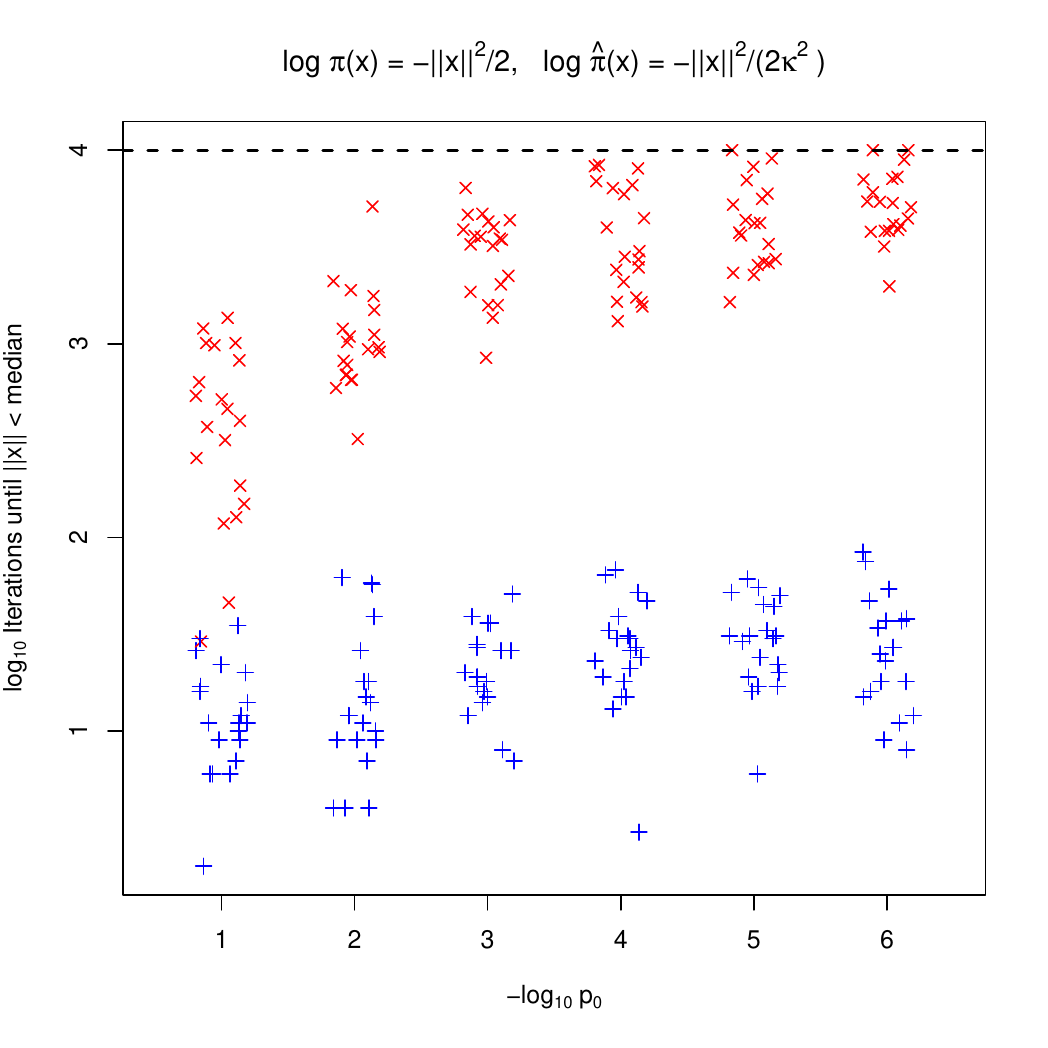}
    \includegraphics[scale=0.45]{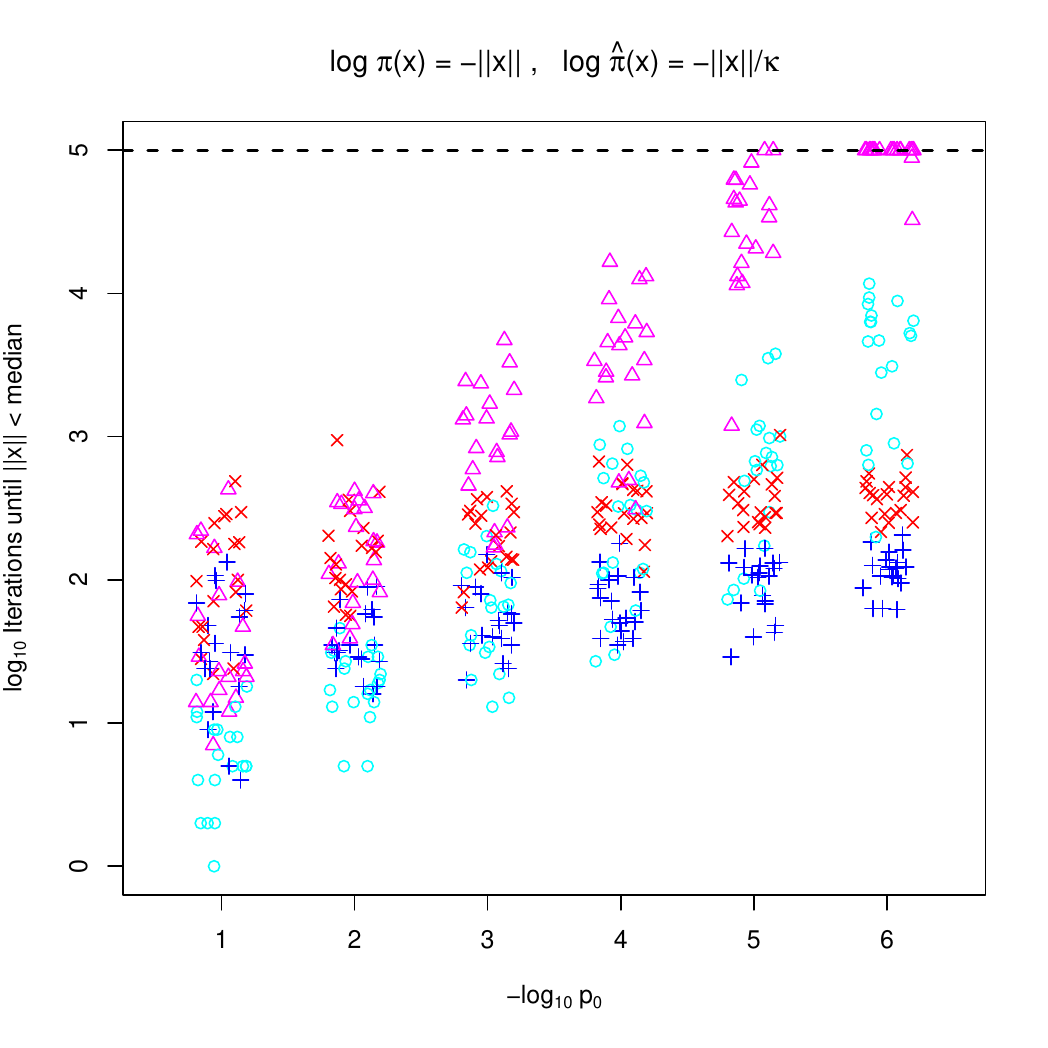}
    \includegraphics[scale=0.45]{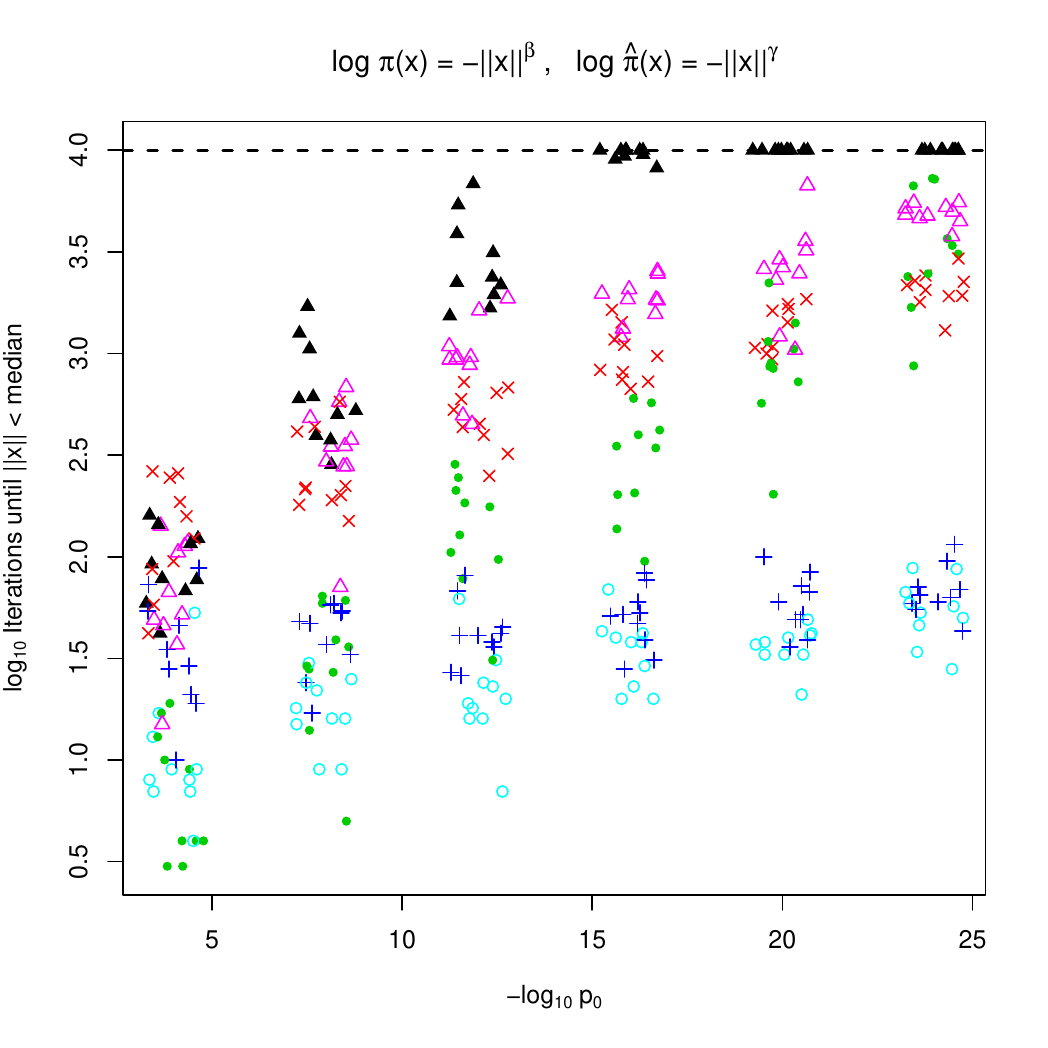}
    \includegraphics[scale=0.45]{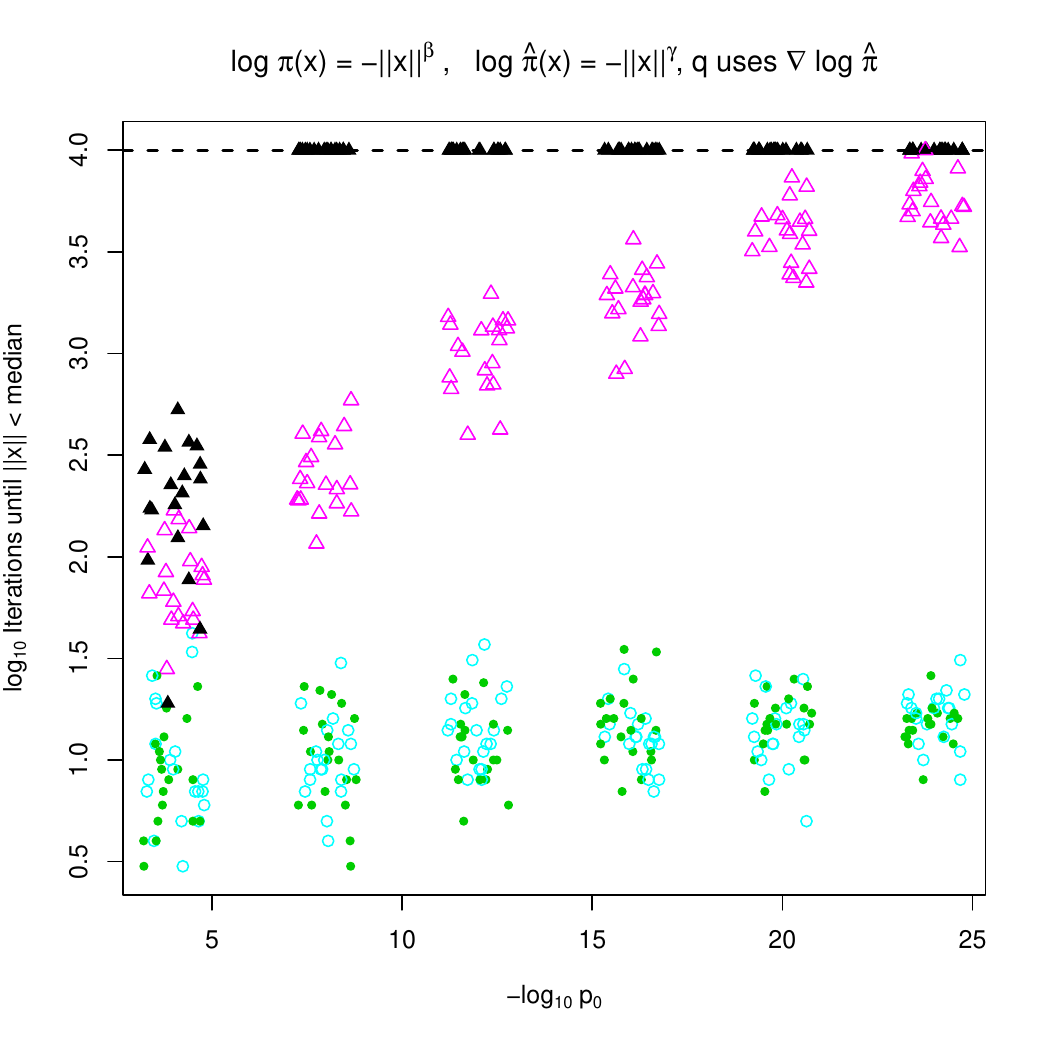}
    \end{center}
\caption{Plots of $\log_{10}$ convergence time against (jittered) $-\log_{10} p_0$ for scenarios (i) top left, (ii) top right, (iii) bottom left and (iv) bottom right. In plots (i)-(iii), a dark blue $+$ corresponds to DARWM where the target has a `good' parameter value (scaling in (i) and (ii), and power in (iii)) and a red $\times$ to DARWM with a `poor' parameter value. In plot (ii) a light blue $\circ$ corresponds to DAIS (DA independence sampler) with $\kappa=2$ and a magenta $\triangle$ to DAIS with $\kappa=1/2$. In plots (iii) and (iv), the light blue $\circ$ corresponds to truncated DATMALA with $\gamma<\beta$, and magenta $\triangle$ to DATMALA with $\gamma>\beta$, whilst a green $\bullet$ corresponds to DAMALA with $\gamma<\beta$, and black $\blacktriangle$ to DAMALA with $\gamma>\beta$.
\label{fig.experiments}}
\end{figure}

\section{Discussion}

Delayed acceptance Metropolis-Hastings algorithms are popular when the posterior is computationally intensive to evaluate yet a cheap approximation is available. Approximations can arise through many mechanisms, including the coarsening of a numerical-integration grid, subsampling from big data, Gaussian process approximation and nearest neighbour averaging. To date, with the exception of \cite{FranksVihola2020} and a note in \cite{BGLR:2015}, little consideration has been given to the properties of the resulting algorithm and, in particular as to whether the delayed-acceptance algorithm might inherit good properties, such as variance bounding, from its parent Metropolis-Hastings algorithm. From the MCMC output, one might reasonably hope to be able to estimate any quantity with a finite variance under $\pi$ and be confident that the Monte Carlo error would reduce in inverse proportion to the square-root of the run length; however, if the algorithm is not variance bounding then this may not be the case.
 
We have investigated the inheritance of the variance bounding property and provided sufficient conditions for it to occur. A general rule of thumb for algorithms with \emph{uniformly local} (see Definition \ref{def.local}) proposals, such as the random walk Metropolis and the truncated MALA, is that the approximation should have \emph{heavier tails} (see Definition \ref{def.heavy.tail}) than the target; however, this is not always necessary (see Example \ref{example.RWMa}). The MALA algorithm does not enjoy the same good properties as the truncated MALA and, in particular, does not necessarily inherit variance bounding even when the approximation does have heavier tails than the target (see Example \ref{example.MALAa}).

A note of caution is also in order: variance bounding (and/or geometric ergodicity) are helpful properties as, in particular, they guarantee the existence of a usual central limit theorem for ergodic averages. However, whilst non-zero, the conductance of a kernel could be exceedingly small (or the geometric rate of convergence execptionally close to one) so that the algorithm might not be useful in practice. Thus, whilst we recommend following the advice in this article when choosing the approximation so as to reduce the chance of false confidence in the resulting Monte Carlo estimates, one should also continue to check other diagnostics, such as trace plots, and to vary any tuning parameters to optimise performance.

\textit{Acknowledgements:} {This paper was motivated by initial conversations with Alexandre
  Thiery about the preservation, and lack of preservation, of
  geometric ergodicity for delayed-acceptance kernels.}

\textit{Data availability:} Data sharing is not applicable to this article as no new data were created or analysed in this study.

\bibliographystyle{apalike}

\bibliography{davb}

\appendix

\section{Proofs of results}
\label{sect.proofs}

\subsection{Proofs of results in Section \ref{sect.generic}}
\subsubsection{Proof of Lemma \ref{lemma.general}}
\label{sect.prove.main}

Since $\epsilon<\kappa_A$ we may define $\beta\in (0,1)$ such that 
$\epsilon=(1-\beta)\kappa_{A}$.
For any $\setA\in \setF$,
\begin{eqnarray*}
\int_{\setA^c}q_B(x,y)\alpha_{B}(x,y)\md y
&\ge&
\int_{\setA^c\cap  \setD(x)}q_B(x,y)\alpha_{B}(x,y) \md y\\
&\ge&
\delta\int_{\setA^c\cap  \setD(x)}q_A(x,y)\alpha_{A}(x,y) \md y
~~\mbox{by \eqref{eqn.genabscts}}.\\
&=&
\delta\left[
\int_{\setA^c}q_A(x,y)\alpha_{A}(x,y) \md y
-
\int_{\setA^c\cap  \setD(x)^c}q_A(x,y)\alpha_{A}(x,y) \md y
\right]\\
&\ge&
\delta\left[
\int_{\setA^c}q_A(x,y)\alpha_{A}(x,y) \md y
-
\int_{\setD(x)^c}q_A(x,y) \alpha_A(x,y)\md y
\right]\\
&\ge&
\delta\left[\int_{\setA^c}q_A(x,y)\alpha_{A}(x,y) \md
  y-(1-\beta)\kappa_{A} \right]~~\mbox{by
  \eqref{eqn.genlocal}}.
\end{eqnarray*}
Integrating both sides over $x \in \setA$ with respect to $\pi$ gives
\begin{eqnarray*}
\pi(\setA)\kappa_{B}(\setA)&=&
\int_{x\in \setA}\int_{y\in \setA^c}\pi(\md x)q_B(x,y)\alpha_{B}(x,y)\md y\\
&\ge&
\delta
\left[\int_{x\in \setA}\int_{y\in\setA^c}\pi(\md
  x)q_A(x,y)\alpha_{A}(x,y) \md y-(1-\beta)\kappa_{A}\int_{x\in
    \setA}\pi(\md x) \right]\\
&=&
\delta
\left[\pi(\setA)\kappa_A(\setA)-(1-\beta)\pi(\setA)\kappa_{A}\right]\\
&\ge&
\beta \delta\pi(\setA)\kappa_{A}.
\end{eqnarray*}
The result follows since only sets with $\pi(\setA)>0$ are relevant. 
 $\square$

\subsubsection{Proof of Proposition \ref{prop.TMALA.conditions}}
Let $\nu(x):=\frac{1}{2}\lambda^2\nabla\log \pi(x)$.\\
(A) (i) 
$\Norm{Y-x}=\Norm{\nu(x)+\lambda Z}\le
\frac{1}{2}\lambda^2D+\lambda\Norm{Z}$, 
where $Z$ is a vector with iid $\mathsf{N}(0,1)$ components. So 
$\Prob{\Norm{Y-x}>r}\le
\Prob{\lambda\Norm{Z}+\frac{1}{2}\lambda^2D>r}\rightarrow 0$ as
$r\rightarrow \infty$.\\
(ii) 
Algebra shows that
\begin{equation}
\label{eqn.deltaq}
\Delta(x,y;q)
=\frac{1}{2\lambda^2}\left\{
2(x-y)\cdot[\nu(x)+\nu(y)]+\Norm{\nu(x)}^2-\Norm{\nu(y)}^2
\right\}.
\end{equation}
Since $\Norm{\nu(x)}\le \lambda^2D/2$,  
 $|\Delta(x,y;q)|\le h(\Norm{y-x}):=D\Norm{y-x} +
\lambda^2D^2/8$, as required.\\
(B) 
Let $Z_*=Z\cdot\nabla \log \pi(x)/\Norm{\nabla \log \pi}\sim N(0,1)$. 
For any $D>0$ we may find $\setA_D\in \setF$ with $\pi(\setA_D)>0$ and 
$\Norm{\nabla\log  \pi(x)}\ge D$ for all $x\in\setA_D$. Hence, for
$x\in\setA_D$,
\[
\Norm{Y-x}=\Norm{\frac{1}{2}\lambda^2\nabla \log \pi(x)+\lambda Z}
\ge \frac{1}{2}\lambda^2\Norm{\nabla \log
  \pi(x)}-\lambda|Z_*|
\ge 
\frac{1}{2}\lambda^2D
-\lambda|Z_*|.
\]
So for any $r>0$, 
$\Prob{\Norm{Y-x}\ge r}\ge
\Prob{|Z_*|\le \frac{1}{2}\lambda D-r/\lambda}$
, which can be made as close to $1$ as desired by taking $D$ to be
sufficiently large. $\square$



\subsubsection{Proof of Theorem \ref{thrm.sideways}}
Since $\kappa_A>0$ and $q_A$ is uniformly local, we may take $\setD(x)=\Bbar(x,r)$, the closure
of $B(x,r)$, where $r$ is chosen so as to satisfy
\eqref{eqn.localA} for $\pi$-almost all $x$, and with
$\epsilon=\kappa_A/2$. 

Next, let
$t=0\wedge (c+b) - 0\wedge (c+a)$.
 Since $0\wedge (c+a)$ is upper bounded by both $0$ and $c+a$, 
if $0<c+b$ then $t\ge 0$, and if
$c+b<0$ then $t\ge b-a$. Thus $t\ge 0\wedge (b-a)$. Hence   
for $y\in \setD(x)$, 
\begin{eqnarray*}
\log \frac{q_B(x,y)\alphaMH_B(x,y)}{q_A(x,y)\alphaMH_A(x,y)}
&=&
\log q_B(x,y) - \log q_A(x,y)\\
&&+ 0 \wedge [\log (\pi(y)/\pi(x)) + \Delta(x,y;q_B)] 
-  0 \wedge [\log (\pi(y)/\pi(x)) + \Delta(x,y;q_A)]\\
&\ge&
\log q_B(x,y) - \log q_A(x,y) + 0 \wedge [\log \Delta(x,y;q_B)-\log
  \Delta(x,y;q_A)]\\
&\ge&
\log q_B(x,y) - \log q_A(x,y) -h(r).
\end{eqnarray*}
The first term is bounded on $\setD(x)$ since 
$\log q_B(x,y) - \log q_A(x,y)$ is continuous, and so 
\eqref{eqn.genabscts} holds and we may apply Lemma
\ref{lemma.general}. Repeat with $A\leftrightarrow B$.  $\square$


\subsection{Proofs of results in Section \ref{sect.backg}}
\label{sec.proof.background}
\subsubsection{Proof of Proposition \ref{prop.GEtonoGEbad}}
Since $\PMH$ is geometrically ergodic, it must have 
a left spectral gap. From \eqref{eqn.var.rep.of.gaps},
\[
\mbox{Gap}_L(P)=1+\inf_{f\in L_0^2(\pi),\braket{f,f}=1}\int \pi(\md x)
P(x,\md y)f(x)f(y)=2-\sup_{f\in L_0^2(\pi),\braket{f,f}=1}\mathcal{E}_P(f),~~~
\]
where the Dirichlet form for the functional $f$ of the Markov chain is
\[
\mathcal{E}_P(f)=\frac{1}{2}\int \pi(\md x)P(x,\md y)[f(x)-f(y)]^2
=
\frac{1}{2}\int \pi(\md
x)q(x,y)\alpha(x,y)[f(x)-f(y)]^2\md y,
\]
for a propose-accept-reject chain.

Since $\alphaDA(x,y)\le\alphaMH(x,y)$, 
$\mathcal{E}_{\PDA}(f)\le \mathcal{E}_{\PMH}(f)$, So 
if $\PMH$ is geometrically ergodic, 
then $\mbox{Gap}_L(\PDA)\ge \mbox{Gap}_L(\PMH)>0$. The result follows as all geometrically ergodic kernels are variance bounding and the only kernels which are variance bounding but not geometrically ergodic have no left-spectral gap, yet we have just shown that $\PDA$ must have a left spectral gap because $\PMH$ is geometrically ergodic.  $\square$

\subsubsection{Proof of Example \ref{example.MHIS}}
Since $\alphaMH(x,y)=1$, the MH algorithm produces iid samples from $\pi$ and so
it is geometrically ergodic (with a spectral gap of $1$) and, hence,
variance bounding . For the DA
algorithm,
\[
\alphaDA(x,y)=\left[1\wedge e^{(k-1)(x-y)}\right]\left[1\wedge e^{(k-1)(y-x)}\right].
\]
For any $r>0$ let 
\[
\setA_r:=\left[{2r},\infty\right),~\setB_r:=\left(0,{r}\right),~\mbox{and}~
  \setC_r:=\left[{r},{2r}\right).
\]
For $(x,y)\in \setA_r\times\setB_r$, $\alphaDA(x,y)\le e^{-|k-1|r}$,
whilst
for $(x,y)\in \setA_r\times\setC_r$, $\alphaDA(x,y)\le 1$. Also,
 $\int_{\setB_r}q(x,y)\md y\le 1$,
whilst 
$\int_{\setC_r}q(x,y)\md y=e^{-kr}-e^{-2kr}\le e^{-kr}$.
Therefore, for $r>\log 2$ (so $\pi(\setA_r)<1/2$) the flow out of $\setA_r$ satisfies
\begin{eqnarray*}
\pi(\setA_r)\kappa_{DA}(\setA_r)
&=&\int_{x\in\setA_r,y\in\setB_r}\pi(x)q(x,y)\alphaDA(x,y)\md y\md x
+
\int_{x\in\setA_r,y\in\setC_r}\pi(x)q(x,y)\alphaDA(x,y)\md y\md x\\
&\le&
e^{-|k-1|r}\int_{x\in\setA_r,}\pi(x)\md x
+
e^{-kr}\int_{x\in\setA_r,}\pi(x)\md x =
\left(e^{-|k-1|r}+e^{-kr}\right)\pi(\setA_r).
\end{eqnarray*}
So, for any $\epsilon>0$ $\exists r$ such that
$\kappa_{DA}(\setA_r)<\epsilon$, and the conductance of the chain is
therefore $0$; the chain is not variance bounding. The lack of geometric ergodicity follows from Proposition \ref{prop.GEtonoGEbad}. $\square$

\subsection{Proofs of results in Section \ref{sect.results}}
\label{sec.proof.results}


\subsubsection{Shorthand for delayed-acceptance kernels}
The following short-hand is used through the  remainder of this section.
\begin{eqnarray}
\label{eqn.sdef}
b_1(x,y)&:=&\log \rDA_1=
\left[\log \pihat(y)-\log \pihat(x)\right]+\Delta(x,y;q),\\
b_2(x,y)&:=&\log \rDA_2=
\left[\log
  \pi(y)-\log \pi(x)\right]-\left[\log \pihat(y)-\log
  \pihat(x)\right].
\end{eqnarray}

\subsubsection{Proof of Theorem \ref{thrm.heavy}}
For any $\epsilon>0$, by \eqref{eqn.localA}, choose $r_\epsilon$ such that 
$\int_{B(x,r_\epsilon)^c}q(x,y)\md y\le \epsilon$; set
$\setD(x):=B(x,r_\epsilon)$  so \eqref{eqn.small.outside} holds.

The `heavier-tail' condition \eqref{eqn.heavy} is equivalent to
$b_1(x,y)+\Delta(x,y;q)\ge 0\Rightarrow b_2(x,y)\ge 0$. 
Applying the identity $[b_1(y,x),b_2(y,x),\Delta(y,x;q)]=-[b_1(x,y),b_2(x,y),\Delta(x,y;q)]$ and relabelling $x$ and $y$ then gives 
$b_1(x,y)+\Delta(x,y;q)\le 0\Rightarrow b_2(x,y)\le 0$.

Next, suppose that $\Norm{x}>r_*$, $\Norm{y}>r_*$ and $y\in\setD(x)$
but $y\notin \setC(x)$, so that
 $b_1(x,y)$ and $b_2(x,y)$ have opposite signs. By
\eqref{eqn.linear.q} $|\Delta(x,y;q)|\le h(r_\epsilon)$, and 
the implications
 derived in the previous paragraph then imply that
 both
$|b_1(x,y)|\le h(r_\epsilon)$ and $|b_2(x,y)|\le h(r_\epsilon)$. Thus
$\setD(x)\cap \setM_{h(r_\epsilon)}(x)\subseteq \setC(x)$.

Finally, let $\Bbar(x,r)$ be the closure of $B(x,r)$, let 
$\mathcal{D}:=\Bbar(0,r_\epsilon+r_*)\times \Bbar(0,2r_\epsilon+r_*)$ and 
let $m_*:=\sup_{(x,y)\in\mathcal{D}}|b_2(x,y)|$;
 since $b_2(x,y)$ is
continuous and $\mathcal{D}$ is compact, $m_*<\infty$.
For $x\in \Bbar(0,r_\epsilon+r_*)$ and $y \in \setD(x)$,
$(x,y)\in\mathcal{D}$ and so $|b_2(x,y)|\le m_*$. For 
$x\in \Bbar(0,r_\epsilon+r_*)^c$ and $y \in \setD(x)$, 
$\Norm{x}>r_*$ and $\Norm{y}>r_*$
and, from the previous paragraph, $\setD(x)\cap \setM_{h(r_\epsilon)}(x)\subseteq \setC(x)$.
 Hence 
 \eqref{eqn.nice.inside} holds with $m=\max(h(r_\epsilon),m_*)$, and the
 result follows from the proof of Corollary \ref{cor.gen.da}. $\square$

\subsubsection{Proof of Proposition \ref{prop.TMALA.GE}}
Let $q(x,y)$ and $q_{\rm RWM}(x,y)$ the be proposal densities for
the Metropolis-Hastings and RWM algorithms, respectively, let
$\alpha(x,y)$ and $\alpha_{\rm RWM}(x,y)$ be the corresponding
acceptance probabilities, and let $v_*:=\mbox{ess sup}~ \Norm{v(x)}$.

Firstly, if $\int_{B(x,r)^c}q_{\rm RWM}(x,y)\md y< \epsilon$ then
$\int_{B(x,r+v_*)^c}q(x,y)\md y< \epsilon$; since $q_{\rm RWM}$ is
uniformly local, so, therefore, is $q$. Next, algebra shows that 
\[
\log q(x,y)-\log q_{\rm RWM}(x,y)=\frac{1}{2\lambda^2}\left[2(y-x)\cdot v(x)+\Norm{v(x)}^2\right].
\]
Now consider $y\in B(x,r)$
and apply the triangle inequality to obtain, 
\[
\Abs{\log q(x,y)-\log q_{\rm RWM}(x,y)}\le \frac{1}{2\lambda^2}\left[2rv_*+v_*^2\right]=:m(r).
\]
Thus $q(x,y)\ge e^{-m(r)}q_{\rm RWM}(x,y)$ and 
$q_{\rm RWM}(x,y)\ge e^{-m(r)}q(x,y)$. 
 Also
$q(x,y)\alpha(x,y)=q(x,y)\wedge
 [q(y,x)\pi(y)/\pi(x)]\ge e^{-m(r)}q_{\rm RWM}(x,y)\alpha_{\rm RWM}(x,y)$
 and, similarly, 
$q_{\rm RWM}(x,y)\alpha_{\rm RWM}(x,y)\ge  e^{-m(r)}q(x,y)\alpha(x,y)$. 
Both implications then follow
 from Theorem \ref{thrm.main} with $\setD(x)=B(x,r)$ and with $r$
 chosen so that both $\int_{B(x,r)^c}q_{\rm RWM}(x,y)\md y\le \epsilon$ and
$\int_{B(x,r)^c}q(x,y)\md y\le \epsilon$.  $\square$

\subsubsection{Proof of Example \ref{ex.TMALAb}}
Let $\PMH_{\rm RWM}$ be the RWM kernel using a Gaussian proposal and targeting
$\pi$, as in Proposition \ref{prop.TMALA.GE}.
Applying Proposition \ref{prop.TMALA.GE} twice shows that $\PMH_{\rm TMALA}$
is variance bounding if and only if $\PMH_{\rm RWM}$ is variance bounding, 
which occurs if and only if $\PMH_{\rm hyp}$ is variance bounding.
The sufficiency of (i) then arises directly from Corollary
\ref{cor.growdiscrep}. For (ii), by Proposition
\ref{prop.TMALA.conditions} A(ii),
applied to the proposal $\qhat$, we
may use Theorem \ref{thrm.heavy}. Geometric ergodicity then follows from Proposition \ref{prop.GEtonoGEbad}. $\square$

\subsubsection{Proofs of Examples \ref{example.MALAa} and \ref{ex.MALAb}}
Let the proposal be 
 $Y=x-\frac{1}{2}\lambda^2\xi x^{\xi-1}+\lambda Z$, where 
$Z\sim\mathsf{N}(0,1)$. Example \ref{example.MALAa} uses $\xi=\beta$
and Example \ref{ex.MALAb} uses $\xi=\gamma$. Now $\nu(x)=-\xi\lambda^2
x^{\xi-1}/2$, so, from \eqref{eqn.deltaq},
\begin{equation}
\label{eqn.deltaqxi}
\Delta^{(\xi)}(x,y;q)=-\frac{\xi}{2}(x-y)(x^{\xi-1}+y^{\xi-1})+\frac{\lambda^2\xi^2}{8}(x^{2\xi-2}-y^{2\xi-2}).
\end{equation}
Also
$Y^k=x^{k}\left(1-\frac{1}{2}\lambda^2\xi x^{\xi-2}+\lambda Z/x\right)^k$. 
Hence, for $k>0$,
\begin{eqnarray*}
\xi<2&\Rightarrow&
U_x:=Y^{\xi-1}/x^{\xi-1}\cip 1, \\
\xi\ge1&\Rightarrow&
V_{x,k}:=\frac{x^k-Y^k}{x^{\xi+k-2}}\cip \frac{1}{2}k\xi\lambda^2-kZ\mathds{1}_{(\xi=1)},
\end{eqnarray*}
as $x\rightarrow \infty$, and where here, and throughout this proof
$\cip$ indicates convergence in probability. 
Now
\[
\Delta^{(\xi)}=
-\frac{1}{2}\xi x^{\xi-1}
\left(\frac{1}{2}\lambda^2\xi x^{\xi-1}-\lambda Z\right)(1+U_x)
+\frac{1}{8}\lambda^2\xi^2x^{3\xi-4}V_{x,2\xi-2}.
\]
However, if $\xi<2$, $x^{3\xi-4}/x^{2\xi-2}\rightarrow 0$ as $x\rightarrow
\infty$, so, for $1\le\xi<2$,
\[
T_{x,\xi}:=\frac{\Delta^{(\xi)}}{x^{2\xi-2}}\cip
 -\frac{1}{2}\xi^2\lambda^2+\xi\lambda Z\mathds{1}_{(\xi=1)},
\]
as $x\rightarrow \infty$. 
Finally, 
\[
b_1(x,Y)=
x^{\xi+\gamma-2}V_{x,\gamma}+ x^{2\xi-2}T_{x,\xi}
~~~\mbox{and}~~~
b_2(x,Y)=
 x^{\xi+\beta-2}V_{x,\beta}-x^{\xi+\gamma-2}V_{x,\gamma}.
\]
\textit{Example \ref{example.MALAa}} ($\xi=\beta$). Consider the
behaviour of $b_1(x,Y)$ and $b_2(x,Y)$ as $x\rightarrow \infty$.
If  $1\le\gamma<\beta$, $b_1$ is dominated by
$x^{2\beta-2}T_{x,\beta}$ and so $b_1\cip -\infty$. 
If $1\le\beta<\gamma$, $b_1$ is dominated by
$x^{\beta+\gamma-2}V_{x,\gamma}$ and $b_2$ is dominated by
$-x^{\beta+\gamma-2}V_{x,\gamma}$; thus, when $\beta>1$, $b_2\cip -\infty$,
and when $\beta=1$, either $b_1\cip -\infty$ or $b_2\cip -\infty$,
depending on the value of $Z$.
In either case, $\alphaDA(x,Y)\cip 0$ as
$x\rightarrow \infty$. 
Given $\epsilon>0$, we choose $x_*$ such that for
all $x>x_*$
$\Prob{\alphaDA(x_*,Y)>\epsilon}<\epsilon$ and set
$\setA_x:=[x,\infty)$, so that $\kappa(\setA_{x_*})<2\epsilon$.
 But $\epsilon$   can be made as
  small as desired, so $\kappa_{DA}=0$. \\
\textit{Example \ref{ex.MALAb}} ($\xi=\gamma$).
When $1\le\beta<\gamma$, as $x\rightarrow \infty$, $b_2$ is dominated
by $-x^{2\gamma-2}V_{x,\gamma}\cip-\infty$ so $\kappa_{DAb}=0$, by an
analogous argument to that for Example \ref{example.MALAa}.\\
If $1\le\gamma<\beta$ then 
we note that the proof of geometric
ergodicity of the MALA algorithm in Theorem 4.1 of \cite{RobTweed1996MALA} applies to
any algorithm with a proposal of the form $q(x,y)\propto
\exp[-|y-[x+\nu(x)]|^2/2\lambda^2]$. For an irreducible and aperiodic
kernel with a continuous proposal density, such as the one under consideration, geometric ergodicity is therefore
guaranteed provided the following two conditions are satisfied:
\[
\eta:=\lim_{|x|\rightarrow \infty}\inf(|x|-|x+\nu(x)|)>0
~~~\mbox{and}~~~
\lim_{|x|\rightarrow \infty}\int_{\mathsf{R}(x)\cap
  \mathsf{I}(x)}q(x,y)\md y = 0,
\]
where $\mathsf{I}(x):=\{y:|y|\le |x|\}$ is the interior and
$\mathsf{R}(x):=\{y:\alphaMH(x,y)<1\}$ is the region where a rejection
is possible.
Now, $x+\nu(x)=x-\gamma x^{\gamma-1}$ so $\gamma \ge 1\Rightarrow \eta>0$. We
will show that if $\gamma \in [1,2)$, for $x>1$ and $y\in \mathsf{I}(x)$, both $b_1(x,y)\ge 0$ and
$b_2(x,y)\ge 0$, so that $\alphaDA_{b}(x,y)=1$ and hence
$\mathsf{R}(x)\cap\mathsf{I}(x)$ is empty.

Now
$b_2(x,y)=x^{\gamma}(x^{\beta-\gamma}-1)+y^{\gamma}(y^{\beta-\gamma}-1)$,
so if $x\ge y$ and $x\ge 1$ then $b_2(x,y)\ge 0$. Further, from
\eqref{eqn.deltaqxi},
\[
b_1(x,y)=x^\gamma-y^\gamma-\frac{\gamma}{2}(x-y)(x^{\gamma-1}+y^{\gamma-1})
+\frac{\lambda^2\gamma^2}{8}(x^{2\gamma-2}-y^{2\gamma-2}).
\]
The final term is non-negative when $x\ge y$. Directly from the
concavity 
of $f(t)=\gamma t^{\gamma-1}$, we
obtain
\[
x^{\gamma}-y^{\gamma}=\int_y^xf(t)\md t \ge (x-y)\frac{f(x)+f(y)}{2} = \frac{\gamma}{2}(x^{\gamma-1}+y^{\gamma-1}),
\]
so the
sum of the first two terms is also non-negative when $x\ge y$. Hence,
for $x\ge 1$, $\mathsf{I}(x)\cap\mathsf{R}(x)$ is empty, as claimed. $\square$

\subsubsection{Proof of Example \ref{ex.TMALA.not.GE}}
Consider any proposal of the form $Y=x+\lambda^2\nu(x)+\lambda Z$, where $Z\sim \mathsf{N}(0,1)$ and $|\nu(x)|\le \nu_*$ for all $x$.
Firstly,
\[
\log \pihat(Y)-\log \pihat(x) = \frac{1}{\gamma}(x^\gamma-Y^{\gamma}).
\]
Secondly, as $x\rightarrow \infty$,
\begin{align}
\Prob{Y>x/2}=\Prob{\lambda Z>-x/2-\lambda^2 \nu(x)}>\Prob{\lambda Z>-x/2+\lambda^2\nu_*}\rightarrow 1.
\end{align}
Also, for any $\delta>0$,
\begin{align}
\Prob{|Y-x|\ge \lambda \delta}=\Prob{|\lambda^2\nu(x)+\lambda Z|\ge \delta}\ge 1-2\delta/\sqrt{2\pi}.
\end{align}
The intermediate value theorem supplies:
$\log \pihat(Y)-\log \pihat(x)=\eta^{\gamma-1}(Y-x)$ for some $\eta\ge \min(x,Y)$. Given any $\epsilon>0$, set $\delta=\sqrt{2\pi}\epsilon/4$ and choose $x_*$ such that $\Prob{Y>x/2}>1-\epsilon/2$ for all $x>x_*$. Then with probability at most $\epsilon$, $|\log \pihat(Y)-\log \pihat(x)|>(x/2)^{\gamma-1}\{|Y-x|\vee \delta\}$. Next,
\begin{align*}
\Delta(x,Y) &= \log q(Y,x)-\log q(x,Y)
=
\frac{1}{2\lambda^2}\left[\lambda^4\nu(x)^2-\lambda^4\nu(y)^2-2\lambda (Y-x)\{\lambda^2\nu(x)+\lambda^2\nu(y)\}\right]\\
&\le 
\frac{1}{2\lambda^2}\left[\lambda^4\nu_*^2-4\lambda^3 (Y-x)\nu_*\right].
\end{align*}
Hence, $|\log \pihat(Y)-\log \pihat(x)|\rightarrow \infty$ and dominates $\Delta(x,Y)$ in probability as $x\rightarrow \infty$; \emph{i.e.}, $\log r_1(x,Y)\sim \log \pihat(Y)-\log \pihat(x)$ and $|\log r_1(x,Y)|\rightarrow \infty$ in probability as $x\rightarrow \infty$.

After some algebra,
\[
\log r_2(x,Y) = \frac{1}{\gamma}Y^\gamma\left(1-\frac{\gamma}{\beta Y^{\gamma-\beta}}\right)-\frac{1}{\gamma}x^\gamma\left(1-\frac{\gamma}{\beta x^{\gamma-\beta}}\right).
\]
Since $\gamma>\beta$, as $x$ (and hence, $Y$), becomes large, we have
\[
\log r_2(x,Y) \sim - \left\{\log \pihat(Y)-\log \pihat(x)\right\}
\]
in probability as $x\rightarrow \infty$.

Combining these two ideas, $0\wedge \log r_1(x,Y) + 0\wedge \log r_2(x,Y)\rightarrow -\infty$ in probability.  Thus $\mathsf{ess}~\sup_x P(x,\{x\})=1$ and the algorithm cannot be geometrically ergodic by Theorem 5.1 of \cite{RobTweed1996RWM}; by Proposition \ref{prop.GEtonoGEbad} it also cannot be variance bounding.




\end{document}